\documentclass[12pt]{amsart}
\usepackage[utf8]{inputenc}
\usepackage{graphicx}
\usepackage{epstopdf}
\usepackage{inputenc}
\usepackage{enumitem}
\usepackage{amsmath}
\usepackage{xcolor}
\usepackage{amssymb}
\usepackage{dsfont}
\newtheorem{theorem}{Theorem}
\newtheorem*{theorem*}{Theorem}
\newtheorem{definition}{Definition}

\newtheorem{corollary}{Corollary}[section]
\newtheorem{lemma}{Lemma}

\newtheorem*{acknowledgements*}{Acknowledgements}
\newtheorem{remark}{Remark}

\newcommand{\Mod}[1]{\ (\mathrm{mod}\ #1)}
\newcommand{\rad}{\text{rad}}
\makeatletter
\def\blfootnote{\gdef\@thefnmark{}\@footnotetext}
\makeatother

\def\house#1{\setbox1=\hbox{$\,#1\,$}%
\dimen1=\ht1 \advance\dimen1 by 2pt \dimen2=\dp1 \advance\dimen2 by 2pt
\setbox1=\hbox{\vrule height\dimen1 depth\dimen2\box1\vrule}%
\setbox1=\vbox{\hrule\box1}%
\advance\dimen1 by .4pt \ht1=\dimen1
\advance\dimen2 by .4pt \dp1=\dimen2 \box1\relax}

\begin{document}
\title{Correlation of multiplicative functions over function fields}
\author{Pranendu Darbar and Anirban Mukhopadhyay}

\address[Pranendu Darbar]{ISI Kolkata  \\
Baranagar, Kolkata  \\
West Bengal 700108,  India}  

  \email{darbarpranendu100@gmail.com}
  
\address[Anirban Mukhopadhyay]{Institute of Mathematical Sciences, HBNI
\\CIT Campus, Taramani\\
Chennai-600113, India}

\email{anirban@imsc.res.in}

\begin{abstract}
In this article, we study function field analogs of a result of K\'{a}tai \cite{KAT} on asymptotic behaviour of correlation of multiplicative functions. More precisely, if we set  $\mathcal{M}_{n}$ and $\mathcal{P}_{n}$ be the set of all monic polynomials and monic irreducible polynomials of degree $n$ over $\mathbb{F}_q$ respectively then for multiplicative functions $\psi_1, \psi_2 :\mathbb{F}_q[x] \rightarrow \mathbb{U}$, and $A_1, A_2\in \mathbb{F}_q[x]\setminus\{0\}$, and $h_1, h_2\in \mathbb{F}_q[x]$, we obtain asymptotic formula for the following correlation functions 
\begin{align*}
\sum_{f\in \mathcal{M}_{n}}\psi_1(A_1f+h_1) \psi_2(A_2f+h_2), \quad \text{ and } \quad
\sum_{P\in \mathcal{P}_{n}}\psi_1(P+h_1)\psi_2(P+h_2)
\end{align*}
for fixed $q$ and  sufficiently large $n$. We also find an  asymptotic formula of the first correlation function when $\psi_1$ and $\psi_2$ are so called ``Hayes pretentious'' multiplicative functions which lead us to deduce a generalized K\'{a}tai's conjecture over function field. We give a new proof  K\'{a}tai's conjecture over function fields for a multiplicative function (see Klurman et al. \cite{KLR} for different proof). We also prove  K\'{a}tai's conjecture  for pair and triplet of multiplicative functions  whose values lies on the unit circle.
As a consequence towards probabilistic interpretation, we derive the behaviour of the distribution of the sum of additive functions.
\end{abstract}
\maketitle

\section{Introduction}
Consider the polynomial ring $\mathbb{F}_{q}[x]$ over a field $\mathbb{F}_q$ with $q$ 
elements. One of the fruitful analogies in number theory is between the integers 
$\mathbb{Z}$ and the polynomial ring $\mathbb{F}_{q}[x]$.
We will introduce correlation of multiplicative functions over $\mathbb{F}_{q}[x]$ 
after highlighting few well known results of correlation of multiplicative functions 
over the integers. In subsequent
sections of the introduction, we will discuss new results in this paper.

\subsection{Correlation of multiplicative functions over intergers}
Let $f: \mathbb{N}\to \mathbb{C}$ be a multiplicative function. 
Many problem from number theory are connected with asymptotic of the mean 
\[
M_f(x):= \frac{1}{x}\sum_{n\leq x}f(n).
\]

\subsubsection*{The Distance function}
In \cite{GS1}, Granville and Soundararajan defined the ``distance" between two multiplicative functions $f, g:\mathbb{N}\to \mathbb{U}$ as
\[
\mathbb{D}(f, g; y; x):= \bigg(\sum_{y<p\leq x}\frac{1-\Re(f(p)\overline{g(p)})}{p}\bigg)^{1/2}
\]
where $\mathbb{U}=\{z\in \mathbb{C}: |z|\leq 1\}$ and in particular
$\mathbb{D}(f, g;x):= \mathbb{D}(f, g; 1; x)$. In several instances
$\mathbb{D}(f, g; \infty)$ is infinite (for example, $\mathbb{D}(1, \mu; \infty)$ is 
infinite). 
However, if $\mathbb{D}(f, g; \infty)< \infty$ then  $f$ is said to be $g$-{\it pretentious} and the case $\mathbb{D}(f, g; \infty)=\infty$ is known as $g$ {\it non-pretentious}. The most important property of this distance functions is that it satisfied the following triangle inequality:
\[
\mathbb{D}(f, g; x)\leq \mathbb{D}(f, h; x)+\mathbb{D}(h, g; x)
\]
for any functions $f, g, h:\mathbb{N}\to \mathbb{U}$. The theory of multiplicative functions get new direction and have been subsequently developed using this new approach called ``pretentious approach" in last two decades. 

\vspace{2mm}
\noindent
The following theorem of Hal\'{a}sz is one of the important theorems related to 
the asymptotic behaviour of $M_f(x)$ in terms of distance function as $x\to \infty$.

\vspace{2mm}
\noindent
\textbf{Theorem A} (Hal\'{a}sz, 1971)\label{halasz main theorem}
Let $f:\mathbb{N}\to \mathbb{U}$ be multiplicative. Then 
\[
M_f(x)=o(1)
\]
unless there exist $t\in \mathbb{R}$ such that $\mathbb{D}(f, n^{it}, \infty)< \infty$ in which case, as $x\to \infty$ we have
\[
M_f(x)=\frac{x^{it}}{1+it}\prod_{p\leq x}\bigg(1-\frac{1}{p}\bigg)\bigg(\sum_{k\geq 0}\frac{f(p^k)p^{-kit}}{p^k}\bigg)+o(1).
\]
\noindent
Quantitative improvements of Hal\'{a}sz's theorem have been obtained by several authors (for example \cite{GHS1}, \cite{GHS2}, \cite{GS1}). 
As a natural generalization of Hal\'{a}sz's theorem one would like to find asymptotic behaviour of the following  $k$-point correlation function 
\begin{align}\label{general correlation function}
M_x(g_1, \ldots, g_k):= \frac{1}{x}\sum_{n\leq x}g_1(F_1(x))\ldots g_k(F_k(x)), \quad k\geq 2
\end{align}
where $g_j$'s are multiplicative functions with modulus less than or equal to $1$ and $F_j(x)$'s are polynomials with integer coefficients. 

\subsubsection{Non-pretentious world}
 If $g_j=\lambda$,  Liouville's function and $F_j(x)=x+h_j, j=1, 2, \ldots, k$ for distinct natural numbers $h_j$'s then a famous conjecture of Chowla asserts that $M_k(x)=o(1)$ as $x\to \infty$. Chowla’s conjecture remains open for any $h_1, \ldots , h_k$ with $k \geq 2$.
 On the basis of the breakthrough work of Matomaki and Radziwill \cite{MR}, Tao \cite{Tao} proved the following two-point ($ k=2$)  logarithmic averaged Chowla and Elliott conjecture:
 \[
 \sum_{x/w(x)<n\leq x}\frac{\lambda(a_1 n+b_1)\lambda(a_2 n+b_2)}{n}=o(\log w(x)).
 \]
 In general if $g_1$ is “non-pretentious” in the sense that 
 \[
 \inf_{|t|\leq x}\mathbb{D}(g_1(n), \chi(n)n^{it}; x)\to \infty, \text{ as } x\to \infty
 \]
  for all Dirichlet characters $\chi$, then 
 \begin{align}\label{logarithmically chowla}
  \sum_{x/w(x)<n\leq x}\frac{g_1(a_1 n+b_1)g_2(a_2 n+b_2)}{n}=o(\log w(x)),
 \end{align}
where $a_1, a_2$ are natural numbers and  $b_1, b_2$ are distinct non-negative integers such
that $a_1b_2 -a_2 b_1\neq 0$, and  $1\leq w(x)\leq x$ is an arbitrary function of $x$ that goes to infinity as $x\to \infty$.

\vspace{2mm}
\noindent
In recent years, progresses  have been made on various averaged forms of Chowla’s conjecture. For instance,  Matom\"{a}ki, Radziwill and Tao \cite{MRT1} established a version of
Chowla’s conjecture where one performs some averaging in the parameters $h_1, \ldots , h_k$. Also Tao and Ter\"{a}v\"{a}inen \cite{TT} proved a structure theorem  for the logarithmically averaged correlations of multiplicative functions which leads to obtain several new cases of logarithmically averaged Chowla and Elliott conjecture for higher $k$-point correlations and in \cite{TT1}, they extend this result to some cases of the unweighted Elliott conjecture at almost all scales.
 \subsubsection{Pretentious world}
K\'{a}tai \cite{KAT} first studied the asymptotic behaviour of the 
sum \eqref{general correlation function} with some assumptions on $g_j$'s 
and $F_j(x)$'s are special polynomials but did not provide any error term. 
Stepanauskas \cite{ST4} studied the asymptotic formula for sum 
\eqref{general correlation function} with explicit error term when $F_j(x)$'s are 
linear polynomials and $g_j$ are ``close" to $1$ (which is much stronger condition than ``pretend" to $1$). 
In \cite{DAR}, the first author studied the asymptotic behaviour of the sum \eqref{general correlation function} with explicit error term when $F_j$'s are polynomial of degree $\geq 2$ and $g_j$'s  are close to $1$.

\vspace{2mm}
\noindent
In a fine work \cite{KLR1}, Klurman provided an asymptotic formula for the sum \eqref{general correlation function} for two multiplicative functions (also for the same multiplicative function) $f, g : \mathbb{N}\rightarrow \mathbb{U}$  with $\mathbb{D}(f (n), n^{it_1} \chi(n), \infty) < \infty$ and $\mathbb{D}(g(n), n^{it_2} \psi(n), \infty) < \infty$ for some primitive Dirichlet characters $\chi, \psi$. As an application of this result together with Tao's theroem \eqref{logarithmically chowla}, Klurman \cite{KLR1} proved K\'{a}tai conjecture that  if $f : \mathbb{N}\to \mathbb{S}^1$
is completely multiplicative and the consecutive values of f are close to each other in the sense
that  
\[
\sum_{n\leq x}|f(n+1)-f(n)|=o(x)
\]
then $f(n) = n^{it}$ for some real number $t$. In the same article Klurman obtained Erd\"{o}s-Tao discrepancy problem among several other results.

\subsection{Correlation of multiplicative functions over $\mathbb{F}_q[x]$}
Consider the polynomial ring $\mathbb{F}_{q}[x]$ over a field with q elements. Let $\mathcal{M}_{n}$ be the set of all monic polynomials of degree $n$ over $\mathbb{F}_q$, so that $|\mathcal{M}_{n}| =q^n$. Let $\mathcal{P}_{n}$ be the set of all monic irreducible polynomials of degree $n$ over $\mathbb{F}_q$. 

\vspace{2mm}
\noindent
Let $\psi: \mathcal{M}\to \mathbb{C}$ be a multiplicative function. 
A central theme is the asymptotic behaviour of the mean
\begin{align}\label{mean value function over function}
\sigma(n, q; \psi):= \frac{1}{q^n}\sum_{f\in \mathcal{M}_{n}}\psi(f), \quad \text{ as } q^n\to \infty.
\end{align}
\vspace{2mm}
\noindent
There are following three way to study the asymptotic behaviour of sum \eqref{mean value function over function}.
\begin{enumerate}[label=\alph*)]
\item When $n\to \infty$ and $q$ is fixed, which is called ``large degree limit",
\item When $q\to \infty$ and $n$ is fixed, which is called ``large finite field limit",
\item When both $n, q \to \infty$.
\end{enumerate}

 
\vspace{2mm}
\noindent
In \cite{GHS}, Granville et al. initiated the study of mean values of multiplicative functions over $\mathbb{F}_q[x]$ by proving the following quantitative analog of  Theorem \ref{halasz main theorem} in the large degree limit aspect.

\vspace{2mm}
\noindent
\textbf{Theorem B} (Granville et al.)
Let $\psi$ be multiplicative functions on $\mathcal{M}$ with modulus less than or equal to $1$ and $\sigma(n, q; \psi)$ be as defined \eqref{mean value function over function}. Then for all integers $n\geq 2$, we have
\[
|\sigma(n, q; \psi)|\leq 2(2+M)e^{-M},
\]
where 
$
\max_{|z|=\frac{1}{q}}|\Psi^{\perp}(z)|:= 2n e^{-M}
$ and $\Psi^{\perp}$ is corresponding power series which has truncated euler product defined in \cite{GHS} with respect to $n$. 

\vspace{2mm}
\noindent
In the large degree limit,  Klurman \cite{KLR} derived analogs of  Wirsing, Hal\'{a}sz and Hall's theorem on $\mathbb{F}_q[x]$.  Motivated by the study of correlation of multiplicative functions over integers one would also like to find the asymptotic behaviour of the following sums:
\begin{align}\label{correlation function over any polynomial}
S_{k}(n, q):=\sum_{f\in \mathcal{M}_{n}}\psi_1(A_1f+h_1)\ldots \psi_k(A_kf+h_k)
\end{align}
and 
\begin{align}\label{correlation function over prime polynomial}
R_k(n, q):=\sum_{P\in \mathcal{P}_{n}}\psi_1(P+h_1)\ldots \psi_k(P+h_k),
\end{align}
where $\psi_1, \ldots, \psi_k$ are multiplicative functions on $\mathcal{M}$ and $A_j\in \mathbb{F}_q[x]\setminus\{0\}$ and $h_j\in \mathbb{F}_q[x]$ are fixed polynomials for all $j=1, \ldots, k$.

\subsubsection{Large finite field aspect}
In the large finite field limit, one can obtain much better results what can be done in the case of integers.
For example, Carmon and Rudnick \cite{CR} proved function field analog of 
Chowla's conjecture in the large finite field limit.
Also, Bary-Soroker \cite{SORO} proved the  function field analog of the Hardy-Littlewood conjecture  over large finite fields.

\subsubsection{Large degree aspect}  A recent
groundbreaking result of Sawin and Shusterman \cite{SS} established the Chowla conjecture in function fields in the form
\[
\frac{1}{q^n}\sum_{f\in \mathcal{M}_{\leq n}}\mu(f+B_1)\ldots \mu(f+B_k)=o(1)
\]
for any $k\geq  1$ and any distinct $B_1, \ldots , B_k \in \mathbb{F}_q[x]$ in the large field case $q > p^2k^2e^2$, where $p = \text{char}(\mathbb{F}_q)$. They used geometric methods to improve on the function field version of the Burgess bound, and showed that, when restricted to certain
special subspaces, the M\"{o}bius function over $\mathbb{F}_q[x]$ can be mimicked by Dirichlet characters.

\subsubsection*{Non-pretentious world in large degree aspect} See section $2.2.4$ for the notion of ``distance" function over function fields. In contrast with the integer case, a new notion  so called ``short interval character"  (see section $2.2.2$) plays a crucial role to obtain the correlation of multiplicative function over function fields which helps to  classify wider class of multiplicative functions. 

\vspace{2mm}
\noindent
Recently, Klurman et al. \cite{KMT} proved the following two-point logarithmic averaged Chowla and Elliott conjecture over function fields : 
Assume that $\psi_1$ satisfies the {\it non-pretentiousness} assumption
\[
\min_{M\in \mathcal{M}_{\leq W}}\min_{\chi \pmod{M}}\min_{\substack{\xi \text{ short }\\ \text{len}\leq n}}\min_{\theta\in [0, 1]}\mathbb{D}^2\left(\psi_1(P), \chi(P)\xi(P)e^{2\pi i \theta \deg(P)}; N \right)\to \infty,
\]
as $N\to \infty$ for every fixed $W \geq 1$. Then for any fixed $B\in \mathbb{F}_q[x]\setminus \{0\}$,
\begin{align}\label{elliott conjecture function fields}
\frac{1}{N}\sum_{f\in \mathcal{M}_{\leq N}}\frac{\psi_1(f)\psi_2(f+B)}{q^{\deg(f)}}=o(1), \quad \text{ as }  N\to \infty,
\end{align}
 where $\mathbb{D}(\psi_1, \psi_2; N)$ is defined by \eqref{The distance function}.
 
 \vspace{2mm}
 \noindent
In particular, if $\psi_1=\psi_2=\mu$, where $\mu : \mathbb{F}_q[t] \to \{-1, 0, +1\}$ is the M\"{o}bius function, this result is viewed as two-point logarithmically averaged Chowla’s conjecture in function fields. In the same article, Klurman et al. proved K\'{a}tai conjecture over function fields using \eqref{elliott conjecture function fields} and new argument which is different from the proof in the integer setting. Also, very recently, Klurman et al. \cite{KMT2} studied the Erd\"{o}s discrepancy problem over function fields.
\subsection{Pretentious world and main results in large degree aspect}
In this article, we will study the asymptotic behaviour of sums $S_k(n, q)$ and $R_k(n, q)$
in large degree limit and various pretentious aspect.

\subsubsection{Correlation with constant function 1}
Let $\psi_j:\mathcal{M}\to \mathbb{U}$ and $\alpha_j:\mathcal{M}\to \mathbb{C}$ be multiplicative functions such that $\alpha_j=\mu*\psi_j$ for all $j=1, 2$.
For fixed polynomials $A_j\in \mathbb{F}_q[x]\setminus\{0\}$ and $h_j\in \mathbb{F}_q[x]$  for all $j=1, 2$ and $n\geq r$, we define
\begin{align}
\label{main term over monics introduction} &Q(n):=\prod_{\deg P\leq n}\upsilon_P  \quad \text{ and } \quad Q(r, n)=\prod_{r<\deg P\leq n}\upsilon_P,\\
\label{main term over monic irreducibles introduction} &Q^{'}(n):=\prod_{\deg P\leq n}\upsilon_P'  \quad \text{ and } \quad Q'(r, n)=\prod_{r<\deg P\leq n}\upsilon_P'
\end{align}
where
\[
\upsilon_P:=\underset{\substack{\left(P^{m_1},P^{m_2}\right) |(A_1h_2-A_2h_1)}}{\sum_{\substack{m_{1}=0}} ^{\infty}\sum_{\substack{m_{2}=0}}^{\infty}}\frac{\alpha_1(P^{m_1})\alpha_2(P^{m_2})}{q^{\deg ([P^{m_1}, P^{m_2}])}},
\quad \upsilon_P':=\underset{\substack{\left(P^{m_1},P^{m_2}\right) |(h_2-h_1)}}{\sum_{\substack{m_{1}=0}} ^{\infty}\sum_{\substack{m_{2}=0}}^{\infty}}\frac{\alpha_1(P^{m_1})\alpha_2(P^{m_2})}{\Phi[P^{m_1}, P^{m_2}]}
\]
and the Euler phi function over function field is defined by
\[
\Phi(f)=|f|\prod_{P|f}\left(1-\frac{1}{|P|}\right).
\] 
 \noindent
The following theorem gives the asymptotic behaviour of $S_2(n, q)$ with explicit error 
term in large degree limit in most  general situation..
\begin{theorem}\label{main theorem 1}
Let $\psi_1$ and $\psi_2$ be multiplicative functions on $\mathcal{M}$ with modulus less than or equal to $1$. Let $A_1, A_2\in \mathbb{F}_q[x]\setminus \{0\}$ and $h_1, h_2\in \mathbb{F}_q[x]$ with $\deg(h_j)<\deg(A_j)$ such that $(A_1, h_1)=(A_2, h_2)=1$ and $\Delta=A_1h_2-A_2h_1\neq 0$. Suppose that $\gamma:=\deg (\Delta)$ and $A:= \max\{d(A_1), d(A_2))\}$. Then there exists a positive absolute constant $c$ such that for all $n\geq r \geq \gamma$ and for all $\frac{1}{2} <\alpha <1$, we have
\begin{align*}
\frac{S_2(n,q)}{q^n}-Q(n) \ll  q^{d(A)}\left(\mathbb{D}(\psi_1, 1; r, n+ d(A_1))+\mathbb{D}(\psi_2, 1; r, n+d(A_2))\right)\\
+q^{(1-2\alpha)n+2(1-\alpha)A}\exp\Big(\frac{cq^{\alpha r}}{r}\Big)+(rq^r)^{-\frac{1}{2}},
\end{align*}
where $Q(n)$  and $\mathbb{D}(\psi_j, 1; r, n+d(A_j))$ are defined as in \eqref{main term over monics introduction} and \eqref{The distance function} respectively.
\end{theorem}
\noindent
The next theorem gives the asymptotic behaviour of $R_2(n, q)$ with explicit error term in large degree aspect.
\begin{theorem}\label{main theorem 2}
Let $\psi_1$ and $\psi_2$ be multiplicative functions on $\mathcal{M}$ with modulus less than or equal to $1$. Let  $h_1, h_2\in \mathbb{F}_q[x]$ such that $\Delta=h_2-h_1\neq 0$. Suppose that $\gamma:=\deg (\Delta)$. Then there exists a positive absolute constant $c$ such that for all $n\geq r \geq \gamma$ and for all $\frac{1}{2} <\alpha <1$, we have
\begin{align*}
\frac{R_2(n,q)}{|\mathcal{P}_{n}|}-Q'(n) \ll  \mathbb{D}(\psi_1, 1; r, n)+\mathbb{D}(\psi_2, 1; r, n) +n^{-A}\exp\Big(\frac{cq^{\alpha r}}{r}\Big)+(rq^r)^{-\frac{1}{2}}
\end{align*}
where $A>0$ is arbitrary constant and $Q'(n)$ and $\mathbb{D}(\psi_j, 1; r, n)$ are defined as in \eqref{main term over monic irreducibles introduction} and \eqref{The distance function} respectively. .
\end{theorem}
\begin{remark}
Notice that
\begin{align*}
 \upsilon_P= \left\{
	\begin{array}
	[c]{ll}
	\left(1-\frac{1}{|P|}\right)\left(\sum_{m=0}^{\infty}\frac{\psi_1(P^m)+\psi_2(P^m)}{|P|^m}\right)-1 & \text{if}\; \, P\nmid \Delta AC\\
	1+\sum_{m=1}^{\infty}\frac{\psi_1(P^m)-\psi_1(P^{m-1})}{|P|^m} & \text{ if }\, \, P\nmid \Delta, P\nmid A, P|C\\
	1+\sum_{m=1}^{\infty}\frac{\psi_2(P^m)-\psi_2(P^{m-1})}{|P|^m} & \text{ if }\,\, P\nmid \Delta, P| A, P\nmid C\\
	1 & \text { if }\, \, P\nmid \Delta, P|A, P|C,
	\end{array}
	\right.
\end{align*}
 and if $P| \Delta$ but $P\nmid A$ and $P\nmid C$ then 
 \[
 \upsilon_P=1+\sum_{\substack{i=1\\ P^{\Delta_P}\Vert \Delta}}^{\Delta_P}\frac{\alpha_1(P^i)\alpha_2(P^i)}{|P|^i}+\delta_A\sum_{j>i}\frac{\alpha_1(P^i)\alpha_2(P^j)}{|P|^j}+\delta_C\sum_{j>i}\frac{\alpha_1(P^j)\alpha_2(P^i)}{|P|^j},
 \]
 where $\delta_f=0$ when $P|f$ and $\delta_f=1$ otherwise, and $\Delta_P$ is an integer such that $P^{\Delta_P}\Vert \Delta$.\\
 Similar expression can be deduce for $\upsilon_P'$. 
\end{remark}
\begin{remark}
Theorem \ref{main theorem 1} and Theorem \ref{main theorem 2} can be extended for $S_k(n, q)$ and $R_k(n, q), k\geq 3$.
\end{remark}

\subsubsection{Correlation with Dirichlet character}

\begin{theorem}\label{theorem for hayes pretentious}
  Let $\psi_1, \psi_2: \mathcal{M}\to \mathbb{U}$ be multiplicative functions such that $\mathbb{D}(\psi_j, \chi_j e_{\theta_j}; \infty)<\infty$ for some primitive characters $\chi_j$ of conductor $Q_j$ and an angle $\theta_j\in [0, 1]$. Assume that $h_1$ and $h_2$  are fixed polynomials. 
   Then as $n\to \infty$,
  \begin{align*}
  &\frac{1}{|\mathcal{M}_n|} \sum_{f\in \mathcal{M}_n}\psi_1(f+h_1)\psi_2(f+h_2) \\
  &=e^{2\pi i (\theta_1+\theta_2)n}\frac{1}{|[Q_1, Q_2]|} \sideset{}{'}\sum_{\rad(f_j)\mid\, Q_j}\frac{\psi_1(f_1)\psi_2(f_2)}{|[f_1, f_2]|}e_{\theta_1}\left(\frac{f_2}{(f_1, f_2)}\right)e_{\theta_2}\left(\frac{f_1}{(f_1, f_2)}\right)e_{-(\theta_1+\theta_2)}([f_1, f_2])\\
  &\times
    \sum_{h([Q_1, Q_2])}\chi_1\left(h\frac{f_2}{(f_1, f_2)}+D_1)\right)\chi_2\left(h\frac{f_1}{(f_1, f_2)}+ D_2\right) \prod_{\substack{P\in \mathcal{P}}}v_P+ o(1), \, \text{ if } \rad(Q_1)=\rad(Q_2),
  \end{align*}
  otherwise, the sum in left hand side vanishes,
  where 
 $v_P$ are defined as in \eqref{main term over monics introduction}  with $\alpha_j= \psi_j \overline{\chi}_je_{-\theta_j}$, $\Delta=(h_1-h_2)$, $D_1, D_2$ are polynomials such that $D_1f_1-D_2f_2=\Delta$, and $\sideset{}{'}\sum$ runs over all $f_1, f_2$ such that $(f_1, f_2)\mid \, \Delta$, $\frac{Q_2}{(Q_1, Q_2)}\mid  \frac{f_1}{(f_1, f_2)}$ and $\frac{Q_1}{(Q_1, Q_2)}\mid  \frac{f_2}{(f_1, f_2)}$.
  \end{theorem}

\begin{corollary}\label{corollary on different functions}
Let $\psi_1, \psi_2: \mathcal{M}\to \mathbb{U}$ be multiplicative functions such that $\mathbb{D}(\psi_j, \chi_j e_{\theta_j}; \infty)<\infty$ for some primitive characters $\chi_j$ of conductor $Q_j$ and an angle $\theta_j\in [0, 1]$. Let $U=\frac{Q_2}{(Q_1, Q_2)}$ and $V=\frac{Q_1}{(Q_1, Q_2)}$. Then as $n\to \infty$,
\begin{align*}
  &\frac{1}{|\mathcal{M}_n|} \sum_{f\in \mathcal{M}_n}\psi_1(f)\psi_2(f+1) \\
  &=e^{2\pi i (\theta_1+\theta_2)n}e^{2\pi i (V\theta_1+U\theta_2+(\theta_1+\theta_2)UV)}\frac{1}{|[Q_1, Q_2]|\, |UV|}\psi_1(U)\psi_2(V) \\
  &\times
    \sum_{h([Q_1, Q_2])}\chi_1\left(hV+D_1)\right)\chi_2\left(hU+ D_2\right) \prod_{\substack{P\nmid \, [Q_1, Q_2]}}v_P+ o(1), \, \text{ if } \rad(Q_1)=\rad(Q_2),
  \end{align*}
  otherwise, the sum in left hand side vanishes,
  where 
  \[
v_P  =\left(1-\frac{1}{|P|}\right)\left(\sum_{m=0}^{\infty}\frac{(\psi_1 \overline{\chi}_1 e_{-\theta_1})(P^m)}{|P|^m}+\sum_{m=0}^{\infty}\frac{(\psi_2 \overline{\chi}_2 e_{-\theta_2})(P^m)}{|P|^m}\right)-1,
\]
and $D_1, D_2$ are polynomials such that $UD_1-VD_2=1$.
\end{corollary}
\begin{theorem}\label{theorem on Dirichlet character}
  Let $\psi: \mathcal{M}\to \mathbb{U}$ be multiplicative functions such that
   \[
   \mathbb{D}(\psi, \chi e_{\theta}; \infty)<\infty
   \]
    for some primitive character $\chi$ of conductor $Q$ and an angle $\theta\in [0, 1]$.
   Then as $n\to \infty$,
    \begin{align*}
   \frac{1}{|\mathcal{M}_n|} \sum_{f\in \mathcal{M}_n}\psi(f+h_1)\overline{\psi(f+h_2)}=\frac{1}{|Q|}\sum_{\substack{f\mid \, \Delta\\ \rad(f)\mid \, Q}}\frac{|\psi(f)|^2}{|f|}
    \sum_{h(Q)}\chi\left(h+D_1\right)\overline{\chi\left(h+ D_2\right)}\\
    \times \prod_{\substack{P\in \mathcal{P}}}v_P +o(1),
 \end{align*}
 where $v_P$ are defined as in \eqref{main term over monics introduction}  with $\alpha_1=\alpha_2= \psi \overline{\chi}e_{-\theta}$, $\Delta=(h_1-h_2)$, and $D_1, D_2$ are polynomials such that $(D_1-D_2)f=\Delta$.
  \end{theorem}
 
 \subsubsection{Correlation with Hayes character}
 \begin{theorem}\label{main theorem on hayes pretentiousness}
 Assume that $1\leq l\leq \lceil \frac{n-1}{2}\rceil$.  Let $\psi: \mathcal{M}\to \mathbb{U}$ be multiplicative functions such that
   \[
   \mathbb{D}(\psi, \chi \xi e_{\theta}; \infty)<\infty,
   \]
    where $\chi$ is a primitive Dirichlet character of conductor $Q$, $\xi$ is a short interval character of length $l$, and an angle $\theta\in [0, 1]$. Then for fixed polynomials $h_1, h_2\in \mathbb{F}_q[x]$ with $\deg(h_j)\leq l$, as $n\to \infty$,
    \begin{align*}
   \frac{1}{|\mathcal{M}_n|} \sum_{f\in \mathcal{M}_n}\psi(f+h_1)\overline{\psi(f+h_2)}=\frac{1}{|Q|}\sum_{\substack{f\mid \, \Delta\\ \rad(f)\mid \, Q}}\frac{|\psi(f)|^2}{|f|}
    \sum_{h(Q)}\chi\left(h+D_1\right)\overline{\chi\left(h+ D_2\right)}\\
    \times \prod_{\substack{P\in \mathcal{P}}}v_P +o(1),
    \end{align*}
    where $v_P$ are defined as in \eqref{main term over monics introduction}  with $\alpha_1=\alpha_2= \psi \overline{\chi}\overline{\xi}e_{-\theta}$, $\Delta=(h_1-h_2)$ and $D_1, D_2$ are polynomials such that $(D_1-D_2)f=\Delta$.
  \end{theorem}
\subsection{Applications}
\subsubsection{Outcomes of Theorem \ref{main theorem 1} and Theorem \ref{main theorem 2}}
The following corollary is a direct application of Theorem \ref{main theorem 1} and Theorem \ref{main theorem 2}.
\begin{corollary}
 Assume the hypothesis of Theorem \ref{main theorem 1} and Theorem \ref{main theorem 2}. Suppose that $\psi_1$ and $\psi_2$ are {\it pretend} to $1$. Then as $n\to \infty$,
\[
\frac{S_2(n, q)}{q^n}\to \prod_{P\in \mathcal{P}}v_P \quad \text{ and } \quad \frac{R_2(n, q)}{|\mathcal{P}_{n}|}\to \prod_{P\in \mathcal{P}} v_P'
\]
where $v_P$ and $v_P'$ are defined as in \eqref{main term over monics introduction} and \eqref{main term over monic irreducibles introduction} respectively.
\end{corollary}
\noindent
We define the truncated Liouville function over function field by
\[
\lambda_{y}(P^{t})=
\begin{cases}
(-1)^{t} \quad (=\lambda(P^t)) \ & \text{if } \deg P\leq y \\
1 \ & \text{if } \deg P>y .
\end{cases} 
\]
It is very interesting to establish
\[
\sum_{f \in \mathcal{M}_{n}}\lambda_y(f)\lambda_y(f+h)=o(q^n), \quad \text{ as } n\to \infty.
\]
Note that, if $y=n$ then it is Chowla's conjecture over function fields in large degree limit. 

\noindent
For very small choice of $y$ the following theorem gives a truncated variant of Chowla's conjecture in large degree limit which is an application of Theorem \ref{main theorem 1}.
\begin{corollary}\label{trancated liouville function}
There is a positive absolute constant $C$ such that if $n\geq 2$, $2\leq y\leq \log n$ and fixed $h\in \mathbb{F}_q[x]$ with $\deg h\leq y$,
then
\[
\bigg|\sum_{f\in \mathcal{M}_{n}}\lambda_y(f)\lambda_y(f+h)\bigg|<C\frac{\log ^4y}{y^4}q^n.
\]
\end{corollary}
\begin{remark}
Mangerel \cite{Mangerel} proved a number field analog of the above theorem with wide range. Following this it may be possible to extend the range of $y$ satisfying  $\frac{n}{y}\to \infty$. 
\end{remark}
\vspace{2mm}
\noindent
Let $\mathbb{F}_q^{*}$ be the group of units in $A:=\mathbb{F}_q[x]$.
Let $\mathcal F_k$ be the set of monic polynomials in $\mathbb{F}_q[x]$ which are
$k$-th power free. 
As a direct application of Theorem \ref{main theorem 1}, we get an asymptotic formula for two simultaneously $k$-free
monic polynomials.
\begin{corollary}
Let $a \in \mathbb{F}_q^{*}$ . Then we have
\[
\frac{1}{q^n}\sum_{\substack{f\in \mathcal{M}_{n} \\ f, \, f+a\in \mathcal{F}_k}}1 =\prod_{P}\bigg(1-\frac{2}{q^{k\deg P}}\bigg)+ O\left(\frac{1}{n^B}\right)
\]
for any $0<B<1$. 
\end{corollary}

\noindent
We can also apply Theorem \ref{main theorem 1} and Theorem \ref{main theorem 2} to $ \frac{\Phi(f)}{|f|}$.
\begin{corollary}
For a fixed $a\in \mathbb{F}_q^{*}$ we have
\[
\frac{1}{q^n}\sum_{f\in \mathcal{M}_{n}}\frac{\Phi(f)\Phi(f+a)}{|f||f+a|}=\prod_{P}\left(1-\frac{2}{q^{2\deg P}}\right)+O\left(\frac{1}{n^B}\right)
\]
and
\[
\frac{1}{|\mathcal{P}_{n}|}\sum_{P\in \mathcal{P}_{n}}\frac{\Phi(P)\Phi(P+a)}{|P||P+a|}= \prod_{P}\bigg(1-\frac{2}{q^{\deg P}(q^{\deg P}-1)}\bigg)+O\bigg(\frac{1}{(\log n)^B}\bigg)
\]
for any $0<B<1$.
\end{corollary}

\subsubsection{Outcomes of Theorem \ref{main theorem on hayes pretentiousness}}
\begin{corollary}\label{main theorem on hayes character}
  Let $\psi: \mathcal{M}\to \mathbb{U}$ be multiplicative functions such that
   \[
   \mathbb{D}(\psi, \chi \xi e_{\theta}; \infty)<\infty,
   \]
    where $\chi$ is a primitive Dirichlet character of conductor $Q$, $\xi$ is a short interval character of length $l$, and an angle $\theta\in [0, 1]$. Then as $n\to \infty$,
    \[
   \frac{1}{|\mathcal{M}_n|} \sum_{f\in \mathcal{M}_n}\psi(f)\overline{\psi(f+1)}=\frac{\mu(Q)}{|Q|}\prod_{P\, \nmid \, Q}\bigg(2\left(1-\frac{1}{|P|}\right)\bigg(\sum_{m=0}^{\infty}\frac{\Re\left((\psi \overline{\chi} \overline{\xi} e_{-\theta})(P^m)\right)}{|P|^m}\bigg)-1\bigg) +o(1).
    \]
  \end{corollary}

\begin{theorem}[K\'{a}tai Conjecture]\label{Katai conjecture}
     Let $\psi: \mathcal{M}\to \mathbb{S}^1$ be a completely multiplicative function. Suppose that as $n \to \infty$,
     \[
     \sum_{f\in \mathcal{M}_n}|\psi(f+1)-\psi(f)|=o(q^n).
     \]
      Then there exists an angle $\theta\in [0, 1)$ and a short interval character $\xi : \mathcal{M} \to \mathbb{S}^1$ such that
$\psi(f) = \xi(f)e^{2\pi i \theta \deg(f)}$.
     \end{theorem}
     
      \begin{corollary}[K\'{a}tai conjecture for pairs]\label{Katai conjectiure for pair}
      Let $\psi, \eta: \mathcal{M}\to \mathbb{S}^1$ be a completely multiplicative function. Suppose that as $n \to \infty$,
     \[
     \sum_{f\in \mathcal{M}_n}|\psi(f+1)-\eta(f)|=o(q^n).
     \]
      Then there exists an angle $\theta\in [0, 1)$ and a short interval character $\xi : \mathcal{M} \to \mathbb{S}^1$ such that 
$\psi(f)=\eta(f) = \xi(f)e^{2\pi i \theta \deg(f)}$.
     \end{corollary}
     \noindent
    As a direct application of Corollary \ref{Katai conjectiure for pair} we have the following.
    \begin{corollary}[K\'{a}tai conjecture for triplets]
     Let $\psi, \eta, \kappa : \mathcal{M}\to \mathbb{S}^1$ be a completely multiplicative function. Suppose that as $n \to \infty$,
     \[
     \sum_{f\in \mathcal{M}_n}|\psi(f+2)-2\eta(f+1)-\kappa(f)|=o(q^n).
     \]
      Then there exists an angle $\theta\in [0, 1)$ and a short interval character $\xi : \mathcal{M} \to \mathbb{S}^1$ such that 
$\psi(f)=\eta(f) = \kappa(f)= \xi(f)e^{2\pi i \theta \deg(f)}$.
    \end{corollary}
    
\subsubsection{Probabilistic viewpoint over $\mathbb{F}_q[x]$}
The asymptotic behaviour of the $k$-point correlation functions $S_k(n, q)$ and $R_k(n, q)$ over $\mathbb{F}_q[x]$ are used to get the behaviour of the distribution of the sum
\begin{align*}\label{general sum of additive functions over integers}
\eta_1(A_1f+h_1)+\ldots +\eta_k(A_2f+h_2),
\end{align*}
and \[
\eta(P+h_1)+\ldots+\eta(P+h_k)
\]
where $\eta_1, \ldots, \eta_k$ are real-valued additive functions, $A_j\in \mathbb{F}_q[x]\setminus \{0\}$ and $h_j\in \mathbb{F}_q[x]$.

\begin{theorem}\label{distribution of additive function over function field introduction}
Let $t, x \in \mathbb{R}$.   Let $A_1, A_2\in \mathbb{F}_q[x]\setminus \{0\}$ and $h_1, h_2\in \mathbb{F}_q[x]$ such that $(A_1, h_1)=(A_2, h_2)=1$ and $\Delta=A_1h_2-A_2h_1\neq 0$.  Assume that $\eta_1$ and $\eta_2$  be real-valued additive functions on $\mathcal{M}$ such that following series converge:
\begin{align*}
\sum_{\substack{|\eta_i(P)|\leq 1}}\frac{\eta_i(P)}{q^{\deg P}},
\quad \quad 
\sum_{\substack{|\eta_1(P)|\leq 1\\ |\eta_2(P)|\leq 1}}
\frac{\eta_1(P)+\eta_2(P)}{q^{\deg P}}, \quad
\sum_{\substack{|\eta_i(P)|> 1}}q^{-\deg P}
\quad \forall i=1, 2.
\end{align*}
Then the distribution functions 
\[
\frac{1}{|\mathcal{M}_{n}|}\Big|\left\{f\in \mathcal{M}_{n}: \, \eta_1(A_1f+h_1)+\eta_2(A_2f+h_2)\leq x\right\}\Big|
\] 
and 
\[
\frac{1}{|\mathcal{P}_{n}|}\Big|\left\{P\in \mathcal{P}_{n}: \, \eta_1(P+h_1)+\eta_2(P+h_2)\leq x \right\}\Big|
\]
converges weakly towards a limit distribution whose characteristic functions  are equal to $ \prod_{P\in \mathcal{P}}v_P$ and $\prod_{P\in \mathcal{P}}v'_P$
where $v_P$ and $v_P'$ are defined as in \eqref{main term over monics introduction} and \eqref{main term over monic irreducibles introduction} respectively with $\psi_j$ replaced by $\exp(it\eta_j), \forall j=1, 2$.
\end{theorem}

\noindent
As a direct consequence of Theorem \ref{distribution of additive function over function field introduction}, we get the following corollary.
\begin{corollary}
Let $z, t \in \mathbb{R}$ and $a \in \mathbb{F}_q^{*}$.
The distribution functions
\[
\frac{1}{|\mathcal{M}_{n}|}\Big|\{f \in \mathcal{M}_{n}: \frac{\Phi(f)\Phi(f+a)}{|f||f+a|}\leq e^{z}\Big| \] 
and 
 \[ \frac{1}{|\mathcal{P}_{n}|}\Big|\{P \in \mathcal{P}_{n}: \frac{\Phi(P)\Phi(P+a)}{|P||P+a|}\leq e^{z}\Big|
\]
converge weakly towards limit distributions. The characteristic functions of these limit distributions are
\[
\prod_{\deg P}\bigg(1+\frac{2\left(\left(1-q^{-\deg P}\right)^{it}-1\right)}{q^{\deg P}}\bigg) \quad \text{ and } \quad \prod_{\deg P}\bigg(1+\frac{2}{q^{\deg(P)}-1}\bigg(\left(1-q^{-\deg P}\right)^{it}-1\bigg)\bigg)
\]
respectively.
\end{corollary}
\section{Preliminaries}
\subsection{Notation} We start by fixing a finite field $\mathbb{F}_{q}$ of odd cardinality $q=p^r$, $r\geq 1$ with a prime $p$. We denote by $\mathbb{A}=\mathbb{F}_{q}[x]$ the polynomial ring over $\mathbb{F}_{q}$.
 For a polynomial $f$ in $\mathbb{F}_{q}[x]$, it's degree will be denoted by either $\deg(f)$ or $d(f)$. 
 
 \vspace{2mm}
\noindent 
 The set of all monic polynomials and monic irreducible polynomials of degree $n$ are denoted by $\mathcal{M}_{n,q}$ (or simply $\mathcal{M}_{n}$ as we fix $q$) and $\mathcal{P}_{n, q}$ (or simply $\mathcal{P}_{n}$) respectively. Let $\mathcal{M}=\cup_{n\geq 1} \mathcal{M}_{n}$ and $\mathcal{P}=\cup_{n\geq 1} \mathcal{P}_{n}$. we also denote the set of all monic polynomials and monic irreducible polynomials of degree less or equal to $n$ by $\mathcal{M}_{\leq n,q}$ (or simply $\mathcal{M}_{\leq n}$) and $\mathcal{P}_{\leq n,q}$ (or simply $\mathcal{P}_{\leq n}$) respectively. 
Let $\mathcal{H}_{n}$ denotes the set of monic square-free polynomials of degree $n$. Observe that for $n\geq 1$, $|\mathcal{M}_{n}|=q^{n}$.  If $f$ is a non-zero polynomial $\mathbb{F}_{q}[t]$, we define the norm of $f$ to be $|f|=q^{d(f)}$. If $f=0$, we set $|f|=0$.

\vspace{2mm}
\noindent
Given polynomials $f, g\in \mathbb{F}_q[x]\setminus \{0\}$, their greatest common divisor is denoted by $(f, g)$ and least common multiple is denoted by $[f, g]$ and defined by $[f, g]=\frac{fg}{(f, g)}$. For a polynomial $f\in \mathbb{F}_q[x]$, we write $e_{\theta}(f):=e(\theta \deg(f))=e^{2\pi i \theta \deg(f)}$. Further, we use $\mathbb{U}:= \{z\in \mathbb{C}: \, |z|\leq 1\}$ and $\mathbb{S}^1:=\{z\in \mathbb{U}:\, |z|=1\}$.

\subsection{Background on Function fields}
We say that $\psi: \mathcal{M}\to \mathbb{C}$ is {\it multiplicative} if $\psi(fg)=\psi(f)\psi(g)$ whenever $(f, g)=1$ and {\it additive} if $\psi(fg)=\psi(f)+\psi(g)$ whenever $(f, g)=1$.

 \subsubsection{Short intervals over function fields}
 Let $B\in \mathbb{F}_q[x]$. For $l\geq 1$, define
 \[
 I(B; l):= \{f\in \mathcal{M}: \, \deg(f-B)< l\}.
 \]
 In other words, 
 \[
 I(B; l)=B+ \widetilde{\mathcal{P}}_{\leq l-1},
 \]
 where $\widetilde{\mathcal{P}}_{\leq l}=\{g\in \mathbb{F}_q[x]: d(g)\leq l\}$. Hence $\#I(B; l)=q^{l}$.
 
 \vspace{2mm}
 \noindent
 Note that for $B$ monic, the interval $I(B; l)$ consists of only monic polynomials. Also, all monic polynomials of degree $n$ are contained in one of the intervals $I(B; l)$ with $B$ monic of degree $n$. Moreover, for $B_1, B_2\in \mathcal{M}_n$ and $l<n$,
 \[
 I(B_1; l)\cap I(B_2; l)\neq \emptyset \Leftrightarrow d(B_1-B_2)< l \Leftrightarrow I(B_1; l)=I(B_2; l).
 \]
 Therefore, we get a partition of $\mathcal{M}_n$ into disjoint intervals parameterized by $B\in \mathcal{M}_n$:
 \begin{align}\label{disjoint intervals}
 \mathcal{M}_n=\bigsqcup_{B\in \mathcal{B}}I(B; l),
 \end{align}
 where $\mathcal{B}=\{B=t^n+b_{n-1}t^{n-1}+\ldots + b_{l}t^{l}: \, b_j\in \mathbb{F}_q\}$.
 
 \subsubsection{Hayes characters}  Let $l\geq 1$ and $Q\in \mathcal{M}$. Define a relation $\mathcal{R}_{Q, l}$ on $\mathcal{M}$ as follows: if $A, B\in \mathcal{M}$ then 
 \begin{align*}
 A\equiv B \pmod{\mathcal{R}_{Q, l}} \text{ if and only if } A\equiv B \pmod{Q} \text{ and }\\
 \text{ the leading } l+1 \text{ coefficients of } A \text{ and }
 B \text{ are the same.}
  \end{align*}
  If $A, B\in \mathcal{M}_n$ then the  later condition  is equivalent to $\deg(A-B)< N-l$.
  
  \vspace{2mm}
\noindent
  An element of $\mathcal{M}$ is invertible $\Mod{\mathcal{R}_{Q, l}}$ if and only if it is co-prime to $Q$. The units of $\mathcal{M}/ \mathcal{R}_{Q, l}$ form an abelian group, denoted by $\left(\mathcal{M}/ \mathcal{R}_{Q, l}\right)^{\times}$. Thus the characters of $\left(\mathcal{M}/ \mathcal{R}_{Q, l}\right)^{\times}$ can be extended to $\mathcal{M}$ be defining them to be zero on non-unit elements. These extentions are Hayes characters (for more details, see \cite{KMT}). 
  Define,
  \[
  G( \mathcal{R}_{Q, l})=\Big\{\widetilde{\chi}: \widetilde{\chi}\in  \widehat{\left(\mathcal{M}/\mathcal{R}_{Q, l}\right)^{\times}}\Big\}.
  \]
  Any Hayes character $\widetilde{\chi} \in G(\mathcal{R}_{Q, l})$ can be decomposed as a product $\chi_Q \xi_l$, where $\chi_Q$ is a Dirichlet character modulo $Q$, and $\xi_l$ is a short interval character of length $l$ (where $l=\max\{v:\, \deg(A-B)<N-v\}$). Notice that the group has size $\phi(Q)q^{l}$.  Note that taking $Q=1$, for any $f\in I(B; l)$ and any short interval character $\xi$ of length $l$, we have $\xi(f)=\xi(B)$. Also $\xi\in G(\mathcal{R}_{1, l})$ means a Hayes character with trivial Dirichlet part.
  
  \vspace{2mm}
  \noindent
    A short interval character $\xi_l$ is called primitive if it is not equal to a short interval character of length strictly larger than $l$. A Hayes character $\widetilde{\chi} \in G(\mathcal{R}_{Q, l})$ is called primitive if both $\chi_Q$ and $\xi_l$ are primitive. 
   Also a Hayes character is called non-principal if either is non-principal in the Dirichlet character or if the length of its short interval character is non-zero. The Hayes conductor of $\widetilde{\chi} \in G(\mathcal{R}_{Q, l})$ is defined by $\text{Cond}(\widetilde{\chi}):= \text{Cond}(\chi_Q)+ \text{len}(\xi_l)=\deg(Q)+l.$
   
   \subsubsection{Orthogonality of Hayes characters}
   The orthogonality relation are given by (see \cite{GORO})
   \[
   \frac{1}{\Phi(Q)q^l}\sum_{A (\mathcal{R}_{Q, l})}\widetilde{\chi_1}(A)\overline{\widetilde{\chi_2}(A)}=\mathds{1}_{\widetilde{\chi_1}=\widetilde{\chi_2}}
   \]
and
\[
\frac{1}{\Phi(Q)q^l}\sum_{\widetilde{\chi} \in G(\mathcal{R}_{Q, l})}\widetilde{\chi}(A)\overline{\widetilde{\chi}(B)}= \mathds{1}_{A\equiv B(\mathcal{R}_{Q, l})}.
\]
\noindent
Let $\xi_1$ and $\xi_2$ be short interval characters of length $l$. The orthogonality (\cite{KMT}, equation $20$) relation with $Q=1$ implies that
  \[
  \frac{1}{q^l}\sum_{A\in \mathcal{R}_{1, l}}\xi_1(A)\overline{\xi_2(A)}=\mathds{1}_{\xi_1=\xi_2}.
  \]
  If $A_1, A_2\in \mathcal{M}_n$ then $A_1\equiv A_2 \pmod{\mathcal{R}_{1, l}}$ if and only if $\deg(A_1-A_2)<n-l$ (first $(l+1)$ coefficients of $A_1$ and $A_2$ coincide).\\
  Note that polynomials of the form 
  \[
  t^n+a_{n-1}t^{n-1}+\ldots + a_{n-l}t^{n-l}
  \]
  represent classes modulo $\mathcal{R}_{1, l}$ and so $\mathcal{B}$ defined earlier, comprises of exactly these polynomials. Each class modulo $\mathcal{R}_{1, l}$ can be written as 
  \[
  I(B; n-l)=B+ \widetilde{\mathcal{P}}_{\leq n-l-1}
  \]
  with $B\in \mathcal{B}$.

\subsubsection{Pretentiousness in function fields} 
Following Klurman \cite{KLR}, we define the ``distance" between two multiplicative 
 functions $\psi_1, \psi_2 : \mathcal{M} \rightarrow \mathbb{U}$ by
 \begin{align}\label{The distance function}
 \mathbb{D}(\psi_1, \psi_2; m, n)=\bigg(\sum_{\substack{P\in \mathcal{P}\\ m\leq \deg P\leq n}} \frac{1-\Re(\psi_1(P)\overline{\psi_2(P)})}{q^{\deg P}}\bigg)^{\frac12}
 \end{align}
 and $\mathbb{D}(\psi_1, \psi_2; n):= \mathbb{D}(\psi_1, \psi_2; 1, n).$
If $\mathbb{D}(\psi_1, \psi_2; \infty)< \infty$ then $\psi_1$ is said to be $\psi_2$-{\it pretentious}, otherwise it is called {\it non-pretentious}. In this ways, if for some $\theta\in [0, 1]$, $\mathbb{D}(\psi, \chi_Q\, \xi_l\, e_{\theta}; \infty)<\infty$ then $\psi$ is called {\it Hayes pretentious}, where $\chi_Q$ is a Dirichlet character modulo $Q$, and $\xi_l$ is a short interval character of length $l$ and $e_{\theta}(P)=e^{2\pi i\theta d(P)}$. More precisely,
  \begin{align*}
  \mathbb{D}(f, \chi_Q\, \xi_l\, e_{\theta}; \infty)=\bigg(\sum_{P\in \mathcal{P}}\frac{(1-\Re(f(P)\overline{\chi}_Q(P)\overline{\xi}_l(P)e^{-2\pi i\theta d(P)}))}{q^{d(P)}}\bigg)^{\frac{1}{2}}\quad \text{ is finite}.
  \end{align*}
  Observe that for any Hayes character $ \chi_Q \xi_l\in G(\mathcal{R}_{Q, l})$, and for some $\theta\in [0, 1]$, 
  \[
  \mathbb{D}(f; \chi_Q\, \xi_l\, e_{\theta}; \infty)<\infty \quad \text{ if and only if } \quad  \mathbb{D}(f \overline{\chi}_Q\overline{\xi}_l e_{-\theta}, 1; \infty)< \infty.
  \]

\subsection{Basic lemmas}
\noindent 
 The following lemma follows from Chinese remainder theorem over function fields.
\begin{lemma}\label{Chinese remainder theorem}
Let $A_1, A_2, g_1, g_2 \in \mathbb{F}_q[x]\setminus \{0\}$ and $h_1, h_2\in \mathbb{F}_q[x]$ such that $(A_1, h_1)=(A_2, h_2)=1$. The congruence system 
\[
A_j f+h_j \equiv 0 \pmod {g_j} \quad j=1, 2
\]
has a solution if and only if $(g_1, g_2)|(A_1h_2-A_2h_1)$. If the solution exists, it is unique modulo $[g_1, g_2]$.
\end{lemma}
\noindent
We now present the analogue of classical prime number  theorem for polynomials over finite fields (see \cite{ROS}, Theorem $2.2$).
\begin{lemma}[Prime Polynomial Theorem]\label{prime polynomial theorem}
Let $\mathcal{P}_{n}$ denote the number of monic irreducible polynomials in $A$ of degree $n$. Then we have
\begin{align*}
|\mathcal{P}_{n}|=\frac{q^n}{n}+O\left(\frac{q^{\frac{n}{2}}}{n}\right).
\end{align*}
\end{lemma}  
\noindent
The following lemma collects some useful estimates over function field.

\begin{lemma}\label{some useful sums}
\begin{enumerate}[label=\alph*)]
\item\label{item 1} Let, $q>1$ and $\gamma>0$. Then 
\[
\sum_{m\leq n}q^m m^{-\gamma}= O(q^n n^{-\gamma}).
\]

\item We have
\begin{align*}
\sum_{\substack{P \in \mathcal{P} \\ \deg P\leq n}}q^{-\deg P}=\log n +c_1 +O(1/n)
\end{align*}
where  $c_1$ is an absolute constant.
\item Also we have
\begin{align*}
\sum_{\substack{m\deg P\leq n/2 \\ m\geq 1}}q^{m\deg P} = O\left(\frac{q^n}{n}\right) \quad \text{ and } \quad \sum_{\substack{m\deg P\leq n \\ m\geq 1}} q^{-(m+1)\deg P} = O(1).
\end{align*}
\end{enumerate}
\end{lemma}
\begin{proof} The estimates are collected from Section $3.3$ of \cite{KZ}. 
\end{proof}

\begin{lemma}[\cite{AND2}, Lemma $2.2$]\label{main lemma for TLF}
Let $f\in \mathbb{F}_q[x]$. Then 
\[
\prod_{P\, |\, f}\bigg(1+\frac{1}{|P|}\bigg)=O(\log (\deg (f))).
\]
\end{lemma}

\begin{lemma}\label{lower bound for certain sum}
Let $\Delta$ be a polynomial in $\mathbb{F}_q[x]$. Then 
\[
\sum_{\substack{M\in \mathcal{M}_{\leq n} \\ (M,\, \Delta)=1}}\frac{\mu^2(M)3^{\omega(M)}}{|M|}\geq c n^3 \prod_{P\, |\, \Delta}\bigg(1+\frac{3}{|P|}\bigg)^{-1},
\]
where $c$ is an absolute constant.
\end{lemma} 
\begin{proof}
Let us consider 
\[
F(s)=\sum_{M}\frac{\mu^{2}(M)3^{\omega(M)}}{|M|^{s+1}}, \quad \Re(s)>0.
\]
We can write 
\[
F(s)=\sum_{n=1}^{\infty}\frac{H(n)}{q^{ns}}, \quad \text{ where } \quad H(n)=\sum_{M\in \mathcal{M}_{n}}\frac{\mu^{2}(M)3^{\omega(M)}}{|M|}.
\]
Substituting $u=q^{-s}$, we define
\[
\widetilde{F}(u)=\sum_{n=1}^{\infty}H(n)u^n, \quad |u|<1.
\]
On the other hand, from Euler product we get
\begin{align*}
\widetilde{F}(u)=\frac{\widetilde{G}(u)}{(1-u)^3}, \quad \text{ where } \quad \widetilde{G}(u)=G(s)=\prod_{P}\bigg(1+\frac{3}{|P|^{s+1}}\bigg)\bigg(1-\frac{1}{|P|^{s+1}}\bigg)^{3}.
\end{align*}
It is easy to see that $G(u)$ converges for $|u|<1$ and therefore bounded.
Comparing the coefficient of $\widetilde{F}(u)$ there exist constants $c_1, c_2>0$ such that 
 \begin{align}\label{upper and lower bound}
 c_1 n^2\leq H(n)\leq c_2 n^2.
 \end{align}
 Using this, we conclude that 
\[
\prod_{P\, |\, \Delta}\bigg(1+\frac{3}{|P|}\bigg)\sum_{\substack{M\in \mathcal{M}_{\leq n} \\ (M,\, \Delta)=1}}\frac{\mu^2(M)3^{\omega(M)}}{|M|}\geq \sum_{M\in \mathcal{M}_{\leq n}}\frac{\mu^2(M)3^{\omega(M)}}{|M|}=\sum_{m\leq n}H(m)\geq cn^3,
\]
which completes the proof of the lemma.
\end{proof}

\subsection{Lemmas on Dirichelet characters}

We start with the following standard property of a primitive Dirichlet character over function fields.
\begin{lemma}\label{Primitive property}
Let $\chi$ be a primitive Dirichlet character of modulus $Q$ on 
$\mathbb F_q[x]$. Then for any non-constant polynomial $D$
dividing $Q$ satisfying $deg(D)<deg(Q)$, there exists 
$C\equiv 1\mod D$, $(C,Q)=1$ such that $\chi(C)\neq 1$.
\end{lemma}

\begin{lemma}\label{Basic lemma}
Let $\chi$ be a primitive Dirchlet character modulo $P^k$ on 
$\mathbb F_q[x]$, where $P$ is an irreducible polynomial. 
For any $1\le i <k$, we have
\[
\sum_{L \Mod{P^i}} \chi(A+P^{k-i}L)=0.
\]
\proof
Let 
\[
S=\sum_{L(P^i)} \chi(A+P^{k-i}L).
\]
Using Lemma \ref{Primitive property}, we get 
$C\equiv 1 \Mod{P^{k-i}}$ such that $\chi(C)\neq 1$.
Write $C$ as $C=1+BP^{k-i}$ and consequently
\begin{align*}
\chi(C)S = & \chi(1+BP^{k-i})\sum_{L(P^i)} \chi(A+P^{k-i}L)\\
         = & \sum_{L(P^i)}\chi(A+(B+L)P^{k-i}+BLP^{2(k-i)})
\end{align*}
If $k\ge 2i$ then $2(k-i)\ge k$, therefore
\[
\chi(C)S =\sum_{L(P^i)}\chi(A+(B+L)P^{k-i})=S,
\]
which implies $S=0$ as $\chi(C)\neq 1$.
\noindent
The other possibility is $k<2i$ and in that case,
by Lemma \ref{Primitive property}, we can find 
$C\equiv 1 \Mod{P^i}$ such that $\chi(C)\neq 1$.
Writing $C=1+DP^i$ and multiplying $\chi(C)$ with $S$, we obtain
\begin{align*}
\chi(C)S=& \sum_{L(P^i)}\chi(A+P^{k-i}L+ADP^i)\\
        =& \sum_{L(P^i)}\chi(A+P^{k-i}(L+ADP^{2i-k}))\\
        =& \sum_{L(P^i)}\chi(A+P^{k-i}L)=S
\end{align*}
since $2i-k>0$.
But this is not possible unless $S=0$, which completes the proof.
\end{lemma}

\begin{lemma}\label{Single character sum}
Let $\chi$ be a primitive Dirchlet character of conductor $P^m$, 
where $P$ is an irreducible polynomial. Then for any $F\in \mathbb{F}_q[x]\setminus\{0\}$ and $D\in \mathbb{F}_q[x]$, we get
\[
\sum_{h(P^m)}\chi\left(hF+D\right)
=\begin{cases}
  |P|^m\chi(D) & \text{ if } P^m|F, \\
   0           & \text{ else }.
 \end{cases}
\]
\end{lemma}

\proof
If $P$ does not divide then we have a full character sum which vanishes. Therefore we suppose that $P^c||F$ and write $F=P^c L$ with
$L$ coprime to $P$. The polynomial $L$ being invertible, after a change of variable, the left hand side becomes
$$
\sum_{h(P^m)}\chi(hP^c+D).
$$
Note that it is enough to prove that the sum is nonzero only if $c\ge m$. Suppose $c<m$. We can write the variable $h$ modulo $P^c$ as 
$h=x+P^{m-c}y$ where $x$ varies over residues modulo $P^{m-c}$ 
and $y$ over residues modulo $P^c$. Thus
\begin{align*}
\sum_{h(P^m)}\chi(hP^c+D)
&=\sum_{\substack{ x(P^{m-c}) \\ y(P^c)}}
                 \chi((x+P^{m-c}y)P^c+D)\\
&=\sum_{\substack{ x(P^{m-c}) \\ y(P^c)}}
                \chi(xP^c+D)\\
&=|P|^c\sum_{x(P^{m-c})}\chi(xP^c+D)\\
&=0,
\end{align*}
where we use Lemma \ref{Basic lemma} in the last step.

\begin{lemma}\label{Double character sum}
Let $P$ be an irreducible polynomial, $\chi_1$ and $\chi_2$
be primitive Dirchlet characters modulo $P^a$ and $P^b$ respectively. For any polynomials $F_1,F_2\in \mathbb{F}_q[x]\setminus\{0\}$ and $D_1, D_2\in \mathbb{F}_q[x]$,
let
\[
I(\chi_1,\chi_2)=\sum_{h(P^{\max(a,b)})}
\chi_1(hF_1+D_1)\chi_2(hF_2+D_2).
\]
If $a<b$ then $I(\chi_1,\chi_2)\neq 0$ only if $P^{b-a}|F_2$.
\end{lemma}

\begin{proof}
Suppose that $a<b$ and write the variable $h$ modulo $P^b$ as
$h=x+P^a y$ with $x$ varying modulo $P^a$ and $y$ modulo 
$P^{b-a}$. The sum on the left hand side becomes
\[
\sum_{x(P^a)}\chi_1(xF_1+D_1)
\sum_{y(P^{b-a})}\chi_2(xF_2+D_2+P^a yF_2).
\]
Putting $A=xF_2+D_2$, the inner sums
\[
S=\sum_{y(P^{b-a})}\chi_2(A+P^a yF_2).
\]
By Lemma \ref{Basic lemma}, $S$ vanishes if $P$ does not divide $F_2$. Let $P^l|| F_2$ with $l\ge 1$.

\noindent
If $l\ge b-a$, then
$$ S=|P|^{b-a}\chi_2(xF_2+D_2).$$

\noindent
Next suppose $1\le l < b-a$ and write $F_2=P^lt$ with $t$ coprime to $P$. Therefore
$$S=\sum_{y(P^{b-a})}\chi_2(A+P^c y),$$
where $c=l+a$.
Clearly $a<c<b$, so we can write the variable $y$ modulo 
$P^{b-a}$ as $y=u+P^{b-c}v$ with $u$ varying modulo $P^{b-c}$
and $v$ modulo $P^{c-a}$. This gives
$$S=|P|^{c-a}\sum_{u(P^{b-c})}\chi_2(A+P^c u)=0$$
from Lemma \ref{Basic lemma}. 
\end{proof}

\subsection{Brun-Titchmarsh inequality over function fields}
Given a non-constant polynomial $M\in \mathbb{F}_q[x]$ and a polynomial $B$ coprime to $M$, let  $\pi_A(n; M, B)$ denotes the number of primes $P\in \mathcal{P}_{n}$ such that $P\equiv B \pmod M$. The prime polynomial theorem for arithmetic progression (\cite{ROS}, Theorem $4.8$) says that
\begin{align}\label{primes in AP over function field}
\pi_A(n; M, B)= \frac{q^n}{n\Phi(M)}+O\left(\frac{q^{\frac{n}{2}}}{n}\right).
\end{align}
As in classical case, we want to allow $\deg(M)$ to grow with $n$. The interesting range of parameter is $\deg (M)< n$ because if $\deg (M)\geq n$ there is at most one prime polynomial of degree $n$ in arithmetic progression $h\equiv B \pmod M$. From \eqref{primes in AP over function field}, we see that if $\frac{n}{2}\leq \deg (M)< n$ then error term becomes larger than main term. Therefore, we must assume that $\deg (M)<\frac{n}{2}$. 

\vspace{2mm}
\noindent
The following lemma, Brun-Titchmarsh inequality over function field which is a special case of a theorem of Chin-Nung Hsu \cite{HSU} gives an upper bound when $\deg(M)<n$.
\begin{lemma}[\cite{HSU}, Theorem $4.3$]\label{Brun Titchmarsh inequality}
Let $\pi_A(n; M, B)$ be defined as above and $\Phi(M)$ denotes the number of coprime residues modulo $M$. Then for $\deg(M)<n$, we have
\[
\pi_A(n; M, B)\leq \frac{2q^n}{\Phi(M)(n-deg (M)+1)}.
\] 
\end{lemma}

\subsection{Application of Selberg sieve over $\mathbb{F}_q[x]$}
The following lemma is an application of Selberg sieve method for polynomials over finite field to estimate $\pi_A(n,M,B)$ on an average when 
$\frac{n}{2}<\deg(M)<n$.
\begin{lemma}\label{useful lemma for Turan kubilius}
Using the above notations, we have
\begin{align*}
\Theta(n):=\sum_{\frac{n}{2}<\deg Q \leq n}\Phi(Q)\pi_A^2(n; Q, -h)\ll |\mathcal{P}_{n}|^2
\end{align*}
where the summation varies over all monic irreducible polynomial $Q$ and $h$ is a fixed polynomial with $\deg h<n$.
\end{lemma}
\begin{proof}
Expanding square, we obtain
\begin{align*}
\Theta(n)=\sum_{\frac{n}{2}<\deg Q \leq n}\Phi(Q)\bigg(\sum_{\substack{P\in \mathcal{P}_{n}\\ P\equiv h(Q)}}1\bigg)
\bigg(\sum_{\substack{P'\in \mathcal{P}_{n}\\ P'\equiv h(Q)}}1\bigg)=\sum_{\substack{A, B\in \mathcal{M}_{\leq \frac{n}{2}}\\ \deg A=\deg B}}S(A, B),
\end{align*}
where 
\[
S(A, B):=\sum_{\substack{\frac{n}{2}<\deg Q\leq n\\ AQ+h\in \mathcal{P}_{n}\\ BQ+h\in \mathcal{P}_{n}}}\Phi(Q).
\]
\noindent
Now we have to find upper bound of the set $S(A, B)$. We define the following sets.
\begin{align*}
&\mathcal{A}:=\Big\{a_M:=M(AM+h)(BM+h):\deg (M)=n-\deg (A)\Big\},\\
&\mathcal{P}_{\Delta}:=\Big\{P\in \mathcal{P}: \deg (P)< [n/2], P\nmid \Delta\Big\},
\end{align*}
where $\Delta= AB(Ah-Bh).$
For a monic polynomial $D\in\mathbb{F}_q$, let us define
\[
\varrho (D)= \# \Big\{M \Mod D: a_M\equiv 0\Mod D\Big\}.
\]
Also let 
\[
\widetilde{Q}=\prod_{\substack{P\in \mathcal{P}_{\Delta}\\ \deg P\leq \frac{n}{2}}}P \quad \text{ and } \quad \mathcal{D} = \left\{D\in \mathcal{M}: D\, |\, \widetilde{Q}, \deg (D)< \frac{n}{5}\right\}.
\]
Observe that 
\begin{align}\label{selberg sieve inequality}
S(A, B)\leq \sum_{\substack{M\in \mathcal{M}_{n-\deg (A)}\\ a_M\in \mathcal{A}\\ (a_M,\, \widetilde{Q})=1}}|M|=q^{n-\deg(A)}\sum_{\substack{M\in \mathcal{M}_{n-\deg (A)}\\ a_M\in \mathcal{A}\\ (a_M, \, \widetilde{Q})=1}}1.
\end{align}

\noindent
Let $X_D$ be real numbers corresponding to each $D$ with $D\in \mathcal{D}$ and  $X_1=1$. 
 We use Theorem $1$ of Webb \cite{WEBB} to obtain 
 \[
 S(A, B)\leq \frac{q^{2n-2\deg (A)}}{Q}+O\left(q^{n-\deg (A)}\sum_{D_1, D_2\in \mathcal{D}}|X_{D_1}X_{D_2}|\varrho_{[D_1, D_2]}\right),
 \]
 where
 \[
 Q=\sum_{D\in \mathcal{D}}\frac{1}{g(D)}, \quad g(D)=f(D)\prod_{P\mid D}\left(1-\frac{1}{f(P)}\right)
 \]
 with $f(D)=\frac{|D|}{\varrho(D)}$.
\noindent
Since $g(D)$ is a multiplicative function on the divisors of $\widetilde{Q}$, then we have
\[
\sum_{M|D}\frac{1}{g(M)}=\frac{f(D)}{g(D)}, \quad D\in \mathcal{D}.
\]
We see that $X_D\ll 1$ for all $D\in \mathcal{D}$. The above $O$-term is bounded above by 
\[
\ll q^{n-\deg(A)}\sum_{D_1, D_2\in \mathcal{D}}3^{[D_1, D_2]}\ll q^{\frac{7n}{5}-\deg(A)}\prod_{\deg (P)\leq \frac{n}{5}}\left(1-\frac{3}{|P|}\right)^{-2}\ll n^6q^{\frac{7n}{5}-\deg(A)}.
\]
Therefore, contribution of $O$-term to $\Theta(n)$ is bounded above by
\[
\ll n^6 q^{\frac{7n}{5}}\sum_{\substack{A, B\in \mathcal{M}_{\leq \frac{n}{2}}\\ \deg A=\deg B}}\frac{1}{|A|}\ll n^6 q^{\frac{19n}{10}},
\]
which is quite small.
Using Lemma \ref{lower bound for certain sum}, we have
\begin{align*}
Q&=\sum_{M\in \mathcal{D}}\frac{1}{g(M)}\geq \sum_{M\in \mathcal{D}}\frac{1}{f(D)}=\sum_{\substack{M\in \mathcal{M}_{\leq\frac{n}{5}} \\ (M,\, \Delta)=1}}\frac{\mu^2(M)3^{\omega(M)}}{|M|}\geq cn^3 \prod_{\substack{P\, \mid \, \Delta}}\bigg(1+\frac{3}{|P|}\bigg)^{-1}
\end{align*}
where $c>0$ is an absolute constant.
Combining above results we get
\[
S(A, B)\ll \frac{q^{2n-2\deg (A)}}{n^3}\prod_{P\, \mid \, \Delta}\bigg(1+\frac{3}{|P|}\bigg). 
\]
Therefore, 
\[
\Theta(n)\ll \frac{q^{2n}}{n^3}\sum_{\substack{A, B\in \mathcal{M}_{\leq \frac{n}{2}}\\\deg A=\deg B}}q^{-2\deg A}\prod_{P|AB(A-B)h}\left(1+\frac{3}{|P|}\right).
\]
We write
\begin{align*}
&\sum_{\substack{A, B\in \mathcal{M}_{\leq \frac{n}{2}}\\\deg A=\deg B}}q^{-2\deg A}\prod_{P|AB(A-B)h}\left(1+\frac{3}{|P|}\right)\\
&=\sum_{\substack{A, B\in \mathcal{M}_{\leq \frac{n}{2}}\\\deg A=\deg B}}q^{-2\deg A}\prod_{P|A}\left(1+\frac{3}{|P|}\right)\prod_{P|B(A-B)}\left(1+\frac{3}{|P|}\right)\prod_{P|h}\left(1+\frac{3}{|P|}\right).
\end{align*}
We use that $\prod_{P|h}\left(1+\frac{3}{|P|}\right)\ll 1$
with constant depending on $q$ and $h$.
Also we find that
\begin{align*}
&\sum_{\substack{A, B\in \mathcal{M}_{\leq \frac{n}{2}}\\\deg A=\deg B}}q^{-2\deg A}\prod_{P|A}\left(1+\frac{3}{|P|}\right)\prod_{P|B(A-B)}\left(1+\frac{3}{|P|}\right)\\
&=\sum_{A\in \mathcal{M}_{\leq \frac{n}{2}}}q^{-2\deg A}\prod_{P|A}\left(1+\frac{3}{|P|}\right)\sum_{\substack{B\\ \deg B=\deg A}}\prod_{P|B(A-B)}\left(1+\frac{3}{|P|}\right).
\end{align*}
 The inner sum becomes
 \begin{align*}
& \sum_{\substack{B\\ \deg B=\deg A}}\prod_{P|B(A-B)}\left(1+\frac{3}{|P|}\right)\leq \sum_{\substack{B\\ \deg B=\deg A}}\prod_{P|B}\left(1+\frac{3}{|P|}\right)\prod_{P|(A-B)}\left(1+\frac{3}{|P|}\right)\\
 &=\sum_{\substack{B\\ \deg B=\deg A}}\sum_{D_1|B}\frac{\mu^2(D_1)3^{\omega(D_1)}}{|D_1|}\sum_{D_2|A-B}\frac{\mu^2(D_2)3^{\omega(D_2)}}{|D_2|}\\
 &= \sum_{D_1, D_2}\frac{\mu^2(D_1)\mu^2(D_2)3^{\omega(D_1)}3^{\omega(D_2)}}{|D_1D_2|}\sum_{\substack{B\in \mathcal{M}_{\deg A, q}\\ B \equiv 0 (D_1)\\ B\equiv A (D_2)}}1.
 \end{align*}
 We observe that $(D_1, D_2)|A$. Let $D=(D_1, D_2)$ and writing $D_i= DF_i$ we have $(F_i, D)=1, (F_1, F_2)=1$ and $\omega(D_i)=\omega(D)+\omega(F_i)$ for all $i=1, 2$. 

\vspace{2mm}
\noindent 
Using \eqref{upper and lower bound}, we have 
 \begin{align*}
 &\sum_{\substack{B\\ \deg B=\deg A}}\prod_{P|B(A-B)}\left(1+\frac{3}{|P|}\right)\\
 &\leq \sum_{D|A}\frac{\mu^2(D)3^{2\omega(D)}}{|D|^2}\sum_{\substack{F_i\in \mathcal{M}_{\leq \deg A-\deg D}\\ (F_1, F_2)=1\\ (F_i, D)=1}}\frac{\mu^2(F_1)\mu^2(F_2)3^{\omega(F_1)+\omega(F_2)}}{|F_1F_2|}\sum_{\substack{B'\in \mathcal{M}_{\deg A-\deg D}\\ B'\equiv 0(F_1)\\ B'\equiv \frac{A}{D}(F_2)}}1\\
 & = \sum_{D|A}\frac{\mu^2(D)3^{2\omega(D)}}{|D|^2}\sum_{\substack{F_i\in \mathcal{M}_{\leq \deg A-\deg D}\\ (F_1, F_2)=1\\ (F_i, D)=1}}\frac{\mu^2(F_1)\mu^2(F_2)3^{\omega(F_1)+\omega(F_2)}}{|F_1F_2|}\bigg(\frac{q^{\deg A-\deg D}}{|F_1F_2|}+O(1)\bigg)\\
 & = q^{\deg A}\sum_{D|A}\frac{\mu^2(D)3^{2\omega(D)}}{|D|^3}\sum_{\substack{F_i\in \mathcal{M}_{\leq \deg A-\deg D}\\ (F_1, F_2)=1\\ (F_i, D)=1}}\frac{\mu^2(F_1)\mu^2(F_2)3^{\omega(F_1)+\omega(F_2)}}{|F_1F_2|^2}\\
 &+O\bigg(\sum_{D|A}\frac{\mu^2(D)3^{2\omega(D)}}{|D|^2}\sum_{\substack{F_i\in \mathcal{M}_{\leq \deg A-\deg D}\\ (F_1, F_2)=1\\ (F_i, D)=1}}\frac{\mu^2(F_1)\mu^2(F_2)3^{\omega(F_1)+\omega(F_2)}}{|F_1F_2|}\bigg)\\
 &\ll q^{\deg A}.
\end{align*}
Hence, we conclude that
\begin{align*}
&\Theta(n)\ll \frac{q^{2n}}{n^3}\sum_{A\in \mathcal{M}_{\leq \frac{n}{2}}}q^{-deg A}\prod_{P|A}\left(1+\frac{3}{|P|}\right)\\
&=\frac{q^{2n}}{n^3}\sum_{A\in \mathcal{M}_{\leq \frac{n}{2}}}\sum_{D|A}\frac{\mu^2(D)3^{\omega(D)}}{|D|}\ll \frac{q^{2n}}{n^2}
\end{align*}
which completes proof of the lemma.
\end{proof}

\subsection{Certain estimates for large prime polynomials}
We introduce some sets which will be used to prove Theorem
\ref{main theorem 1} and Theorem \ref{main theorem 2}.
Let $h_1$, $h_2\in \mathbb{F}_q[x]$ and $A_1, A_2\in \mathbb{F}_q[x]\setminus\{0\}$ be fixed such that $\deg(h_k)<\deg(A_k)$.
For any $f\in \mathcal{M}$ and $k=1, 2$, we define
\[
\mathcal{P}_f(k):= \Big\{P\in\mathcal{P}: P^m\| A_k f+h_k \text{ and } |1-\psi_k(P^m)|>\frac{1}{2}\Big\}.
\] 
For $r<n$, we consider the following sets:
\[
\mathcal{N}_r = \Big\{f\in \mathcal{M}_{n}: \exists k\in \{1, 2\} \text{ and }  \exists  P\in \mathcal{P}_f(k) \text{ with } \deg P>r  \Big\}
\]
and taking $A_k=1$,
\[
\mathcal{Q}_r = \Big\{P\in \mathcal{P}_{n}: \exists k\in \{1, 2\} \text{ and }  \exists  Q\in \mathcal{P}_P(k) \text{ with } \deg Q>r  \Big\}.
\]
\begin{lemma}\label{cardinality-N_r-Q_r}
With notations as above, upper bound for cardinalities of the sets $\mathcal{N}_r$ and $\mathcal{Q}_r$ are as follows.
\[
|\mathcal{N}_r|\ll
q^n\sum_{j=1}^{2}\mathbb{D}(\psi_j, 1; r, n+d(A_j)) +\frac{q^{n-r}}{r}
\]
and
 \[
  |\mathcal{Q}_r|\ll 
  |\mathcal{P}_{n}|\sum_{j=1}^{2}\mathbb{D}(\psi_j, 1; r, n)
  +\frac{|\mathcal{P}_{n}|}{rq^r}+\frac{|\mathcal{P}_{n}|} {q^{\frac{n}{4}}}.
 \]
\end{lemma}

\begin{proof}

Observe that
\begin{align*}
|\mathcal{N}_r| & \ll \sum_{j=1}^{2}\sum_{\substack{f\in \mathcal{M}_{n} \\ P^m \| A_jf+h_j \\ |1-\psi_j(P^m)|>1/2 \\ \deg P >r}}1 \ll q^n\sum_{j=1}^{2}\sum_{\substack{m\deg P\leq n+d(A_j) \\ |1-\psi_j(P^m)|>1/2 \\ \deg P>r}}\frac{1}{q^{m\deg P}}\\
& \ll q^n \sum_{j=1}^{2}\sum_{r<\deg P\leq n+d(A_j)} \frac{|1-\psi_j(P)|}{q^{\deg P}}+q^n \sum_{\deg P>r}q^{-2\deg P}\\
& \ll q^n \sum_{j=1}^{2}\mathbb{D}(\psi_j, 1; r, n+d(A_j))+\frac{q^{n-r}}{r}.
\end{align*}
\noindent
Interchanging summation we get
\begin{align*}
|\mathcal{Q}_r|\ll  &\sum_{j=1}^{2}\sideset{}{^*}\sum_{\substack{k\deg Q\leq n\\ k\geq 1}}\pi_A(n; Q^k, -h_j)=\sum_{j=1}^{2}\sum_{\deg Q\leq n}^{*}\pi_A(n; Q, -h_j)\\
&+\sum_{j=1}^{2}\sideset{}{^*}\sum_{\substack{k\deg Q\leq n\\ k\geq 2}}\pi_A(n; Q^k, -h_j)
=:M_1+M_2,
\end{align*}
where $\sum^{*}$ denotes  sum over $Q\in \mathcal{P}$ satisfying $\deg (Q)>r$ and $|1-\psi_j(Q^k)|>\frac{1}{2}$.

\vspace{2mm}
\noindent
Using Lemma \ref{Brun Titchmarsh inequality} and Cauchy-Schwarz inequality, we have
\begin{align*}
M_1 &\ll \frac{q^n}{n} \sum_{j=1}^{2}\sideset{}{^*}
\sum_{r<\deg Q\leq \frac{n}{2}}\frac{1}{\Phi(Q)}
+\sum_{j=1}^{2}\bigg(\sideset{}{^*}\sum_{\frac{n}{2}<\deg Q\leq n}
\frac{1}{\Phi(Q)}\bigg)^{\frac{1}{2}}\bigg(
\sideset{}{^*}\sum_{\frac{n}{2}<\deg Q\leq n}\Phi(Q)\pi_A^2(n; Q, -h_j)\bigg)^{\frac{1}{2}}\\
&\ll \frac{q^n}{n}\sum_{j=1}^{2}
\sum_{r<\deg Q\leq \frac{n}{2}}\frac{|1-\psi_j(Q)|^2}{\Phi(Q)}
+\sum_{j=1}^{2}\bigg(
\sum_{\frac{n}{2}<\deg Q\leq n}\frac{|1-\psi_j(Q)|^2}{\Phi(Q)}\bigg)^{\frac{1}{2}} 
(\Theta(n))^{\frac{1}{2}}\\
&\ll |\mathcal{P}_{n}| 
\sum_{j=1}^{2}\mathbb{D}^2\left(\psi_j, 1; r, \frac{n}{2}\right)
+|\mathcal{P}_{n}|\sum_{j=1}^{2}\mathbb{D}\left(\psi_j, 1; \frac{n}{2}, n\right),
\end{align*}
where we used Lemma \ref{useful lemma for Turan kubilius} in the second term.  Also
\begin{align*}
M_2 \ll 
\frac{q^n}{n}\sum_{\substack{k\deg Q\leq \frac{n}{2}\\\deg Q>r;\, k\geq 2}}
\frac{1}{\Phi(Q^k)}+q^n
\sum_{\substack{\frac{n}{2}<k\deg Q\leq n\\ k\geq 2}}
\frac{1}{\Phi(Q^k)}\ll \frac{|\mathcal{P}_{n}|}{rq^r}
+\frac{|\mathcal{P}_{n}|}{q^{\frac{n}{4}}}.
\end{align*}
Combining these estimates, we have
\[
|\mathcal{Q}_r|\ll |\mathcal{P}_{n}|\sum_{j=1}^{2}\mathbb{D}(\psi_j, 1; r, n)+\frac{|\mathcal{P}_{n}|}{rq^r}+\frac{|\mathcal{P}_{n}|} {q^{\frac{n}{4}}}.
\]
\end{proof}

\subsection{Variants of Tur\'{a}n-Kubilius inequality over function field}
The following lemma is a shifted version of Tur\'{a}n-Kubilius inequality over function field in large degree limit.

\begin{lemma}\label{Turan-Kubilius inequality}
For a sequences of complex numbers $\{\psi(P^m), P\in \mathcal{P}, m\geq 1\}$, we have
\begin{align*}
\widetilde{S}:=\sum_{f \in \mathcal{M}_{n}}\bigg|\sum_{P^m\| f+h} \psi(P^m)- \sum_{m\deg P \leq n}\frac{\psi(P^m)}{q^{m\deg P}}\left(1-q^{-\deg P}\right)\bigg|^2 \ll q^n\sum_{m\deg P \leq n}\frac{|\psi(P^m)|^2}{q^{m\deg P}}
\end{align*}

\vspace{2mm}
\noindent
where $h$ is some fixed polynomial with $\deg h<n$.
\end{lemma}

\begin{proof}
First we assume that $\psi(P^m)=0$ for all irreducible polynomials $P$ with $m\deg (P)>\frac{n}{2}$.
By opening square of modulus on the left hand side, 
the coefficient of $\psi(P^m)\overline{\psi(Q^r)}$ for distinct irreducible polynomials
$P$ and $Q$, is
\begin{align*}
&\sum_{\substack{ f\in \mathcal{M}_{n} \\ P^m, Q^r \| f+h}} 1
- \sum_{\substack{ f\in \mathcal{M}_{n} \\ P^m \| f+h }}\frac{1-q^{-\deg Q}}{q^{r\deg Q}}
- \sum_{\substack{ f\in \mathcal{M}_{n} \\ Q^r \| f+h }}\frac{1-q^{-\deg P}}{q^{m\deg P}}\\
&+ \sum_{ f\in\mathcal{M}_{n}}
\frac{\left(1-q^{-\deg P}\right)\left(1-q^{-\deg Q}\right)}{q^{m\deg P+r\deg Q}}.
\end{align*}
Observe that
\[
\sum_{f \in \mathcal{M}_{n}}\sum_{P^m, Q^r \| f+h} 1 
= q^n\frac{\left(1-q^{-\deg P}\right)\left(1-q^{-\deg Q}\right)}{q^{m\deg P+r\deg Q}}.
\]
By treating three other sums analogously, we  find that the coefficient of 
$\psi(P^m)\overline{\psi(Q^r)}$ is zero. Therefore only diagonal terms have non-zero coefficients. 
It is easy to see that coefficient of $|\psi(P^m)|^2$ is $\le q^{n-m\deg(P)}$.
Thus the diagonal term is bounded above by
\[
 \le q^n\sum_{m\deg P \leq n}\frac{|\psi(P^m)|^2}{q^{m\deg P}}.
\]
\noindent
If we assume that $\psi(P^m)=0$ for all monic irreducible polynomials 
$P$ with $m\deg P\leq n/2$. Therefore, if $f \in \mathcal{M}_{n}$, 
there exist at most one prime polynomial power $P^m \| f+h$ such that $\psi(P^m)\neq 0$. 
So we have
\begin{align*}
\widetilde{S}\ll q^n\sum_{m\deg P \leq n}\frac{|\psi(P^m)|^2}{q^{m\deg P}}.
\end{align*}

\noindent
Finally, we write a general $\psi$ as $\psi_1+\psi_2$, where $\psi_1(P^m)=0$ for all monic irreducible  polynomials with $m\deg P>n/2$ and $\psi_2(P^m)=0$ with $m\deg P\leq n/2$ and combining above calculation we get the required result.
\end{proof}

\vspace{2mm}
\noindent
As a direct consequence of Lemma \ref{Turan-Kubilius inequality}, using Lemma \ref{some useful sums} and Cauchy-Schwarz inequality twice we get the following version of Tur\'{a}n-Kubilius inequality over function field.

\begin{lemma}\label{useful turan-kubilius inequality}
For a sequences of complex numbers $\{a(P^m), P\in \mathcal{P}, m\geq 1\}$, we have
\[
\sum_{f \in \mathcal{M}_{n}}\bigg|
\sum_{P^m\|f+h}a(P^m)- \sum_{m\deg P \leq n}\frac{a(P^m)}{q^{m\deg P}}\bigg| 
\ll q^n\bigg(\sum_{m\deg P \leq n}\frac{|a(P^m)|^2}{q^{m\deg P}}\bigg)^{1/2}
\]
where $h$ is some fixed polynomial of $\deg h<n$.
\end{lemma} 

\noindent
The following lemma is an analog of Lemma \ref{useful turan-kubilius inequality} for irreducible polynomials.
\begin{lemma}\label{useful turan-kubilius over primes}
Let $h$ be a fixed polynomial of $\deg h<n$. 
For a sequences of complex numbers $\{a(P^m), P\in \mathcal{P}, m\geq 1\}$, we have
\[
\sum_{P\in \mathcal{P}_{n}}\Big|
\sum_{Q^k\|P+h}a(Q^k)-A(n)\Big|\ll |\mathcal{P}_{n}|B(n)
\]
\noindent
where 
\[
A(n):=\sum_{\substack{Q\in \mathcal{P} \\ k\deg Q\leq n}}
\frac{a(Q^k)}{|Q^k|}
\ \text{ and } 
B^2(n):=\sum_{\substack{Q\in \mathcal{P} \\ k\deg Q\leq n}}
\frac{|a(Q^k)|^2}{\Phi(Q^k)}.
\]
\begin{proof}
For $m<n$, using triangle inequality we have
\begin{align}\label{three sum}
&\sum_{P\in \mathcal{P}_{n}}\Big|\sum_{Q^k\|P+h}\psi(Q^k)-A(n)\Big|\leq \sum_{P\in \mathcal{P}_{n}}\Big|\sum_{\substack{Q^k\|P+h\\ k\deg (Q)\leq m}}\psi(Q^k)-A(m)\Big|\\
&+\sum_{P\in \mathcal{P}_{n}}\Big|\sum_{\substack{Q^k\|P+h\\ k\deg (Q)\leq m}}\psi(Q^k)\Big|+\sum_{P\in \mathcal{P}_{n}}|A(n)-A(m)|=: L_1+L_2+L_3, \nonumber
\end{align}
 Using Cauchy-Schwarz inequality and Lemma \ref{prime polynomial theorem}, we get
\[
L_1\leq \bigg(\sum_{P\in \mathcal{P}_{n}}1\bigg)^{\frac{1}{2}}
\bigg(\sum_{P\in \mathcal{P}_{n}}\Big|\sum_{Q^k\|P+h}\psi(Q^k)-A(m)\Big|^2\bigg)^{\frac{1}{2}}
\leq \frac{q^{\frac{n}{2}}}{n^{\frac{1}{2}}} L_{4}^{\frac{1}{2}},
\]
where 
\[
L_4:=\sum_{P\in \mathcal{P}_{n}}\Big|\sum_{Q^k\|P+h}\psi(Q^k)-A(m)\Big|^2.
\]
Note that
\begin{align*}
\sum_{\substack{P\in \mathcal{P}_{n}\\ Q^{k}\|P+h}}1=\pi_A(n, Q^k, -h)-\pi_A(n, Q^{k+1}, -h)
\end{align*}
and 
\begin{align*}
\sum_{\substack{P\in \mathcal{P}_{n}\\ Q_1^{k_1}, Q_2^{k_2}\|P+h}}1 &= \pi_A(n, Q_1^{k_1}Q_2^{k_2}, -h)-\pi_A(n, Q_1^{k_1+1}Q_2^{K_2}, -h)\\
&-\pi_A(n, Q_1^{k_1}Q_2^{k_2+1}, -h)+\pi_A(n, Q_1^{k_1+1}Q_2^{k_2+1}, -h).
\end{align*}
Choosing $m=\frac{n}{4}$, we use \eqref{primes in AP over function field} and by simplifying square of modulus of $L_{4}$, we observe that
\begin{align*}
L_4&=\frac{q^n}{n}\sum_{k\deg Q\leq m}\frac{|\psi(Q^k)|^2}{\Phi(Q^k)}\bigg(1-\frac{1}{q^{\deg Q}}\bigg)\bigg(1-\frac{1}{q^{k\deg Q}}\bigg)\\
&+O\bigg(\frac{q^{\frac{n}{2}+2m}}{nm}\sum_{k\deg Q\leq m}\frac{|\psi(Q^k)|^2}{\Phi(Q^k)}\bigg).
\end{align*}
Thus, we have
\[
L_4\ll \frac{q^n}{n}B^2(n).
\]
The next term of \eqref{three sum} gives
\begin{align*}
L_2 & \le \sum_{n/4 <k\deg Q\le n}|a(Q^k)|\pi(n, Q^k, -h) \\
&=\sum_{n/4 <k\deg Q\le n/2}|a(Q^k)|\pi(n, Q^k, -h)+\sum_{n/2 <k\deg Q\le n}|a(Q^k)|\pi(n, Q^k, -h)\\
&:=L_5+L_6
\end{align*}
It is easy to show using Cauchy-Schwarz inequality that
\[
L_5\ll \frac{q^n}{n} B(n).
\] 
Using Lemma \ref{Brun Titchmarsh inequality}, Lemma \ref{useful lemma for Turan kubilius} and Cauchy-Schwarz inequality,  we have
\begin{align*}
L_6=&\sum_{P\in \mathcal{P}_{n}}\bigg|\sum_{\substack{Q^k\|P+h\\ \frac{n}{2}<k\deg Q\leq n}}a(Q^k)\bigg|\leq \sum_{\substack{\frac{n}{2}<k\deg Q\leq n}}|a(Q^k)|\pi_A(n; Q^k, -h) \\ 
&\ll  \bigg(\sum_{\substack{\frac{n}{2}<k\deg Q\leq n\\ k\geq 1}}
\frac{|a(Q^k)|^2}{\Phi(Q^k)}\bigg)^{\frac{1}{2}} \bigg(\sum_{\substack{\frac{n}{2}<k\deg Q\leq n\\ k\geq 1}}\Phi(Q^k)\pi_A^2(n; Q^k, -h)\bigg)^{\frac{1}{2}}\\
&\ll B(n)\Theta(n)^{\frac{1}{2}}+ B(n)q^n \bigg(\sum_{\substack{\frac{n}{2}<k\deg Q\leq n\\ k\geq 2}}\frac{1}{\Phi(Q^k)}\bigg)^{\frac{1}{2}}\\
&\ll B(n)\frac{q^n}{n} + B(n)\frac{q^n}{q^{n/8}}\ll \frac{q^n}{n} B(n).
\end{align*} 
It is easy to show using Cauchy-Schwarz inequality that
\[
 |A(m)-A(n)|\le \left( \sum_{m<k\deg Q\le n}\frac{|a(Q^k)|^2}{\phi(Q^k)}\right)^{1/2}
 \left(\sum_{m<k\le n}\frac{\phi(Q^k)}{|Q^k|^2}\right)^{1/2}
 \ll B(n)
\]
for any $m<n$ and thus the last term of \eqref{three sum} becomes
\[
L_3\ll \frac{q^n}{n}B(n). 
\]
This completes the proof of lemma.
\end{proof}
\end{lemma}

\subsection{Probabilistic set-up over $\mathbb{F}_q[x]$}
Let $\psi:\mathcal{M}\to \mathbb{R}$ be a real valued additive function. Define  $\Omega:=\mathcal{M}_{n}$, which is a finite set of $q^n$ elements and $\psi_n$ to be the restriction of $\psi$ to $\mathcal{M}_n$. Let $\psi(\Omega)=\{x_1, \ldots, x_l\}$ be an enumeration. The subsets $A_i:=\{f \in \Omega: \psi_n(f)=x_i \}, \quad i=1, \ldots, l$, of $\Omega$ are  pairwise disjoint and form a partition of $\Omega$. The $\sigma$-algebra $\mathfrak{F}$ generated by this partition consists of union of a finite number of subsets $A_i$.
For $A \in \mathfrak{F}$, let $\nu (A)=\frac{|A|}{q^n}$, where $|A|$ is the cardinality of $A$. Then $\nu$ is a probability measure on $\mathfrak{F}$ and $(\Omega, \mathfrak{F}, \nu)$ is a finite probability space. Now $\psi_n$ is a random variable on $(\Omega, \mathfrak{F}, \nu)$. The distribution function of $\psi_n$ is 
\[
\nu_n(\psi, x)=\frac{1}{q^n}\Big|\{f \in \mathcal{M}_{n}: \psi_n(f)\leq x\}\Big|.
\]
 
\noindent
\begin{definition}
If there exists a distribution function $\Psi$ such that $\frac{1}{q^n}\nu_n(\psi, x)$ converges point-wise to $\Psi(x)$ as $n\rightarrow \infty$ ,
then we say that $\psi$ has the limit distribution function $\Psi$.
\end{definition}

\vspace{2mm}
\noindent
Associated with a distribution function $F(x)$, the characteristic function is defined by
\[
\phi(t)=\int_{-\infty}^{\infty}e^{itx}dF(x).
\]
 This characteristic function is defined for all real values of $t$. It is uniformly continuous for $-\infty<t< \infty$ and satisfies $\phi(0)=1$, $|\phi(t)|\leq 1$.

\begin{lemma}[{\cite{TEN}}, Theorem $3$]\label{continuity theorem}
Let $\left\lbrace F_n\right\rbrace_{n=1}^{\infty}$ be a sequence of distribution functions and $\left\lbrace\phi_n\right\rbrace_{n=1}^{\infty}$ be the corresponding sequence of characteristic functions. Then $F_n$ converges weakly to the distribution function $F$ if and only if $\phi_n$ converges pointwise on $\mathbb{R}$ to a function $\phi$ which is continuous at $0.$ Moreover, $\phi$ is the characteristic function of $F$ and the convergence of $\phi_n$ to $\phi$ is uniform on any compact subset.
\end{lemma}

\section{Proof of Theorem \ref{main theorem 1}}
We begin by spliting $\psi_1, \psi_2$ into parts, one is trivial on large primes and other on small primes.
For $r\geq 1$ and $j=1, 2$, we define multiplicative functions $\psi_{jr}$ and $\psi_{jr}^{*}$, by
\begin{align*}
\psi_{jr}(P^m)=
\begin{cases}
\psi_j(P^m) \ & \text{if } \deg P\le r, \\
1 \ & \text{if } \deg P >r,
\end{cases} \quad \text{and} \quad
\psi_{jr}^{*}(P^m)=
\begin{cases}
1 \ & \text{ if } \deg P\le r, \\
\psi_j(p^m) \ & \text{if } \deg P>r.
\end{cases}
\end{align*}
We use M\"{o}bius inversion to define  
\begin{equation*}
\alpha_{jr}(P^m)=
\begin{cases}
\psi_j(P^m)-\psi_j(P^{m-1}) \ & \text{if } \deg P\le r,\\
0 \ & \text{if } \deg P>r,
\end{cases}
\end{equation*}
so that $\psi_{jr}=1\ast \alpha_{jr}$, $j=1, 2$.

\begin{lemma}\label{alpha_jr-sum}
 For each $j=1,2$ and for any $\beta\in (0,1)$, we have
 \[
  \sum_{g\in \mathcal{M}}\frac{|\alpha_{jr}(g)|}{|g|^\beta}
  \ll \exp\left(c\frac{q^{(1-\beta)r}}{r}\right)
 \]
 for some absolute constant $c>0$.
\end{lemma}

\begin{proof}
 Since $\alpha_{jr}$ is multiplicative, we can write
 \[
  \sum_{g\in \mathcal{M}}\frac{|\alpha_{jr}(g)|}{|g|^\beta}
  \le \prod_{\substack{P\in \mathcal{P} \\ \deg(P)\le r}}
  \left(1+\sum_{m=1}^\infty\frac{|\alpha_{jr}(P^m)|}{|P^m|^\beta}\right)
 \]
Recall that $\alpha_{jr}(P^m)=\psi_{jr}(P^m)-\psi_{jr}(P^{m-1})$ and hence 
$|\alpha_{jr}(P^m)|\le 2$.
Therefore the last product is bounded above by
\begin{align*}
 \prod_{\substack{P\in \mathcal{P} \\ \deg(P)\le r}}
  \left(1+\sum_{m=1}^\infty\frac{2}{q^{\beta m\deg(P)}}\right)
 & \le \exp\Bigg(2\sum_{\substack{P\in \mathcal{P} \\ \deg(P)\le r}}
 \sum_{m=1}^{\infty}\frac{1}{q^{\beta m\deg(P)}}\Bigg)\\
 & \le \exp\Bigg( c\sum_{\substack{P\in \mathcal{P} \\ \deg(P)\le r}}
 \frac{1}{q^{\beta\deg(P)}}\Bigg)
\end{align*}
for a suitable $c>0$. This completes the proof using $|\mathcal{P}_{m}|\ll \frac{q^m}{m}$ for
any $m$. 

\end{proof}
\noindent
We write
\begin{align*}
& \frac{S_2(n, q)}{q^n}- Q(n)=Q(r, n)\bigg(\frac{1}{q^n}\sum_{f \in \mathcal{M}_{n}}\psi_{1r}(A_1f+h_1)\psi_{2r}(A_2f+h_2)-Q(r)\bigg) 
\\
& + \frac{1}{q^n}\sum_{f \in \mathcal{M}_{n}}\psi_{1r}(A_1f+h_1)\psi_{2r}(A_2f+h_2)\Big(\psi_{1r}^{*}(A_1f+h_1)\psi_{2r}^{*}(A_2f+h_2)-Q(r, n)\Big).
\end{align*}
\noindent
Observe that  $Q(r, n)\ll 1$. Therefore, we have
\begin{align}
\frac{S_2(n, q)}{q^n}- Q(n)& \ll \Big|\frac{1}{q^n}\sum_{f \in \mathcal{M}_{n}}\psi_{1r}(A_1f+h_1)\psi_{2r}(A_2f+h_2)-Q(r)\Big| \\
& + \frac{1}{q^n}\sum_{f \in \mathcal{M}_{n}}\Big|\psi_{1r}^{*}(A_1f+h_1)\psi_{2r}^{*}(A_2f+h_2)-Q(r, n)\Big| \nonumber.
\end{align}
Now we see that
\begin{align*}
&\sum_{f \in \mathcal{M}_{n}}\psi_{1r}(A_1f+h_1)\psi_{2r}(A_2f+h_2)= \sum_{f\in \mathcal{M}_{n}}\sum_{g_1| A_1f+h_1}\alpha_{1r}(g_1) \sum_{g_2| A_2f+h_2}\alpha_{2r}(g_2) \\
& = \sum_{\substack{g_j\in \mathbb{F}_q[x]\setminus \{0\}\\d(g_j) \leq n+d(A_j)\, \forall j\\ (g_1, g_2)\,\mid \, \Delta \\ P\, \mid \, g_j\Longrightarrow d(P)\leq r}}\alpha_{1r}(g_1)\alpha_{2r}(g_2)\sum_{\substack{f\in \mathcal{M}_{n} \\ g_1 \mid A_1f+h_1 \\ g_2\mid A_2f+h_2}}1 .
\end{align*}
By using Lemma \ref{Chinese remainder theorem}, we have 
\[
\big|\{f \in \mathcal{M}_{n}: A_1f+h_1 \equiv 0 \Mod{g_1}, A_2f+h_2 \equiv 0 \Mod{g_2} \}\big|= \frac{q^n}{|[g_1, g_2]|}+O(1),
\]
whenever $(g_1, g_2)|(A_1h_2-A_2h_1)$.
Therefore we obtain
\begin{align*}
&\sum_{f \in \mathcal{M}_{n}}\psi_{1r}(A_1f+h_1)\psi_{2r}(A_2f+h_2)=q^n\sum_{\substack{g_j\in \mathbb{F}_q[x]\setminus \{0\}\\d(g_j) \leq n+d(A_j)\, \forall j\\ (g_1, g_2)\,\mid \, \Delta \\ P\, \mid \, g_j\Longrightarrow d(P)\leq r}}\frac{\alpha_{1r}(g_1)\alpha_{2r}(g_2)}{|[g_1, g_2]|}+\\
& + O\bigg(\sum_{\substack{d(g_j) \leq n+d(A_j)\\ \forall j=1, 2}}|\alpha_{1r}(g_1)\alpha_{2r}(g_2)|\bigg)=:M_1+E_1.
\end{align*}
So, we have
\begin{align*}
M_1 & = q^n \sum_{\substack{g_j\in \mathbb{F}_q[x]\setminus \{0\} \, \forall j \\ (g_1, g_2)\, \mid \, \Delta\\ P\, \mid \, g_j\Longrightarrow d(P)\leq r}}\frac{\alpha_{1r}(g_1)\alpha_{2r}(g_2)}{|[g_1, g_2]|}+O\bigg(q^n \sum_{\deg (g_1)>n+d(A_1)}\sum_{g_2\in \mathbb{F}_q[x]\setminus \{0\}}\frac{|\alpha_{1r}(g_1)\alpha_{2r}(g_2)|}{|[g_1, g_2]|}\bigg)\\
&=q^n Q(r)+E_2.
\end{align*}
Since $(g_1, g_2)|\Delta$ and $\Delta$ is a fixed polynomial we have $|(g_1, g_2)|\ll 1$ with constant depending on $q, A_j$ and $h_j$. By writing $[g_1, g_2]=\frac{g_1g_2}{(g_1, g_2)}$ we get
\begin{align*}
E_2\ll q^n \sum_{\deg (g_1)>n+d(A_1)}\frac{|\alpha_{1r}(g_1)|}{|g_1|}\sum_{g_2 \in \mathbb{F}_q[x]\setminus\{0\}}\frac{|\alpha_{2r}(g_2)|}{|g_2|}.
\end{align*}
Using Lemma \ref{some useful sums} (b), we observe that 
\begin{align*}
\sum_{g\in \mathbb{F}_q[x]\setminus\{0\}}\frac{|\alpha_{jr}(g)|}{|g|}&\ll\prod_{\deg P\leq r}\bigg(1+\sum_{k=1}^{\infty}\frac{|\alpha_{jr}(P^k)|}{q^{k\deg P}}\bigg)\ll \prod_{\deg P\leq r}\bigg(1+\frac{2}{q^{\deg P}-1}\bigg)\\
&\ll  \exp\Big(c\sum_{\deg P\leq r}q^{-\deg P}\Big)\ll r^{c_1}
\end{align*}
for some constant $c, c_1>0$. For $0<\alpha <1$, using Lemma \ref{some useful sums} (a), we have
\begin{align*}
\sum_{\deg (g)>n+d(A_j)}\frac{|\alpha_{jr}(g)|}{|g|} & \ll \frac{1}{q^{(n+d(A_j))\alpha}}\sum_{g\in \mathbb{F}_q[x]\setminus\{0\}}\frac{|\alpha_{jr}(g)|}{q^{(1-\alpha)\deg (g)}}\\
&\ll \frac{1}{q^{(n+d(A_j))\alpha}}\exp \Big(c_2\sum_{\deg P\leq r}\frac{1}{q^{(1-\alpha)\deg P}}\Big) \\
& \ll \frac{1}{q^{(n+d(A_j))\alpha}}\exp\Big(c_3\sum_{m\leq r}\frac{q^{m\alpha}}{m}\Big)\ll \frac{1}{q^{(n+d(A_j))\alpha}}\exp\Big(c_4\frac{q^{r\alpha}}{r}\Big)
\end{align*}
for constants $c_3>0$ and $c_4>0$.
Using these estimates, we get
\[
E_1\ll q^{(2-2\alpha)(n+A)}\exp \Big(c\frac{q^{r\alpha}}{r}\Big) \quad \text{ and }
\quad E_2 \ll q^{(1-\alpha )n-\alpha d(A_1)} \exp \Big(c\frac{q^{r\alpha}}{r}\Big),
\]
where $A=\max\{d(A_1), d(A_2)\}$.
Finally, we have to calculate the following sum
\[
E_3:= \sum_{f \in \mathcal{M}_{n}}\Big|\psi_{1r}^{*}(A_1f+h_1)\psi_{2r}^{*}(A_2f+h_2)-Q(r, n)\Big|.
\]
We decompose $E_3$ as
\begin{align*}
E_3 &=\sum_{f\in \mathcal{N}_r}\Big|\psi_{1r}^{*}(A_1f+h_1)\psi_{2r}^{*}(A_2f+h_2)-Q(r, n)\Big|\\
&+ \sum_{f\not\in \mathcal{N}_r}\Big|\psi_{1r}^{*}(A_1f+h_1)\psi_{2r}^{*}(A_2f+h_2)-Q(r, n)\Big|
=: E_4+E_5.
\end{align*}
Using Lemma \ref{cardinality-N_r-Q_r}, we get
\[
E_4\ll |\mathcal{N}_r|\ll
q^n\sum_{j=1}^{2}\mathbb{D}(\psi_j, 1; r, n+d(A_j)) +\frac{q^{n-r}}{r}
\] 
\noindent
We recall that if $\Re(u)\le 0, \Re(v)\le 0$, then
\begin{align}
& \label{exp} \left|\exp(u)-\exp(v)\right|\le |u-v|\quad \text{and} \\ 
& \label{logarithm} \log (1+z)=z+O(|z|^2), \quad \text{ if } |z|\le 1, |\arg(z)|\le \frac{\pi}{2} . 
\end{align}
Note that 
\begin{align*}
\log Q(r, n)=\sum_{r<\deg P\leq n}\log \bigg(1+\sum_{j=1}^{2}\sum_{m=1}^{\infty}\frac{\psi_j(P^m)-\psi_j(P^{m-1})}{q^{m\deg P}}\bigg).
\end{align*}
 Using \eqref{logarithm}, 
  \[
 \log \psi_{jr}^{*}(A_jf+h_j)= \sum_{\substack{P^m\| A_jf+h_j \\ \deg P>r}}\left(\psi_j(P^m)-1\right) +O\bigg(\sum_{\substack{P^m\| A_jf+h_j \\ \deg P>r}}|\psi_j(P^m)-1|^2\bigg).
  \] 
  \noindent
Using \eqref{exp} and \eqref{logarithm}, we have
\begin{align*}
&E_5 \ll \sum_{j=1}^{2}\sum_{f\in \mathcal{M}_{n}}\bigg|\sum_{\substack{P^m\| A_jf+h_j \\ \deg P>r}}(\psi_j(P^m)-1)-\sum_{\substack{m\deg P\leq n \\ \deg P>r}}\frac{\psi_j(P^m)-1}{q^{m\deg P}}\bigg| \\
& +\sum_{f\in \mathcal{M}_{n}}\bigg|\sum_{j=1}^{2}\sum_{\substack{m\deg P\leq n \\ \deg P>r}}\frac{\psi_j(P^m)-1}{q^{m\deg P}}-\log Q(r, n)\bigg| +O\bigg(\sum_{f\in \mathcal{M}_{n}}\sum_{\substack{P^m\| A_jf+h_j \\ \deg P>r}}|\psi_j(P^m)-1|^2\bigg)\\
&=: E_6+E_7+E_8.
\end{align*} 
\noindent
We obtain
\begin{align*}
E_8 &\ll q^n\sum_{j=1}^{2}\sum_{\substack{m\deg P\leq n+d(A_j) \\m\geq 1;  \deg P>r}} \frac{|\psi_j(P^m)-1|^{2}}{q^{m\deg P}} \ll q^n \sum_{j=1}^{2}\sum_{r<\deg P\leq n+d(A_j)}\frac{|\psi_j(P)-1|^2}{q^{\deg P}}+\frac{q^{n-r}}{r}\\
& \ll q^n \left(\mathbb{D}^2(\psi_1, 1;r, n+d(A_1))+\mathbb{D}^2(\psi_2, 1;r, n+d(A_2))\right)+\frac{q^{n-r}}{r},\\
 &E_7  = q^n \bigg|\sum_{j=1}^{2}\sum_{\substack{r<\deg P\leq n}}\frac{\psi_j(P)-1}{q^{\deg P}}+O\bigg(\sum_{\deg P>r}q^{-2\deg P}\bigg)-\sum_{j=1}^{2}\sum_{\substack{r<\deg P\leq n}}\frac{\psi_j(P)-1}{q^{\deg P}}\bigg| \\
& \ll q^n \sum_{\deg P>r}q^{-\deg P}\ll \frac{q^{n-r}}{r}.
\end{align*}
  \noindent
Following lines of proof of  Lemma \ref{useful turan-kubilius inequality} for the shifts $A_jf+h_j$, we have
\begin{align*}
E_6&\ll q^{n+d(A)} \bigg(\sum_{j=1}^{2}\sum_{\substack{m\deg P\leq n+d(A_j) \\m\geq 1; \deg P>r}}\frac{|\psi_j(P^m)-1|^2}{q^{m\deg P}}\bigg)^{1/2}\\
&\ll q^{n+d(A)} \sum_{j=1}^{2}\mathbb{D}(\psi_j, 1; r, n+d(A_j))+\frac{q^{n}}{(rq^{r})^{\frac{1}{2}}}.
\end{align*}

\noindent
Combining the above estimates, we get the theorem.

\section{Proof of Theorem \ref{main theorem 2}}
We use the functions $\psi_{jr}^* (j=1, 2)$ defined in Section $3$. Writing analogously, we get
\begin{align*}
&\frac{R_2(n, q)}{|\mathcal{P}_{n}|}-Q'(n)=Q'(r, n)\bigg(\frac{1}{|\mathcal{P}_{n}|}\sum_{P\in \mathcal{P}_{n}}\psi_{1r}(P+h_1)\psi_{2r}(P+h_2)-Q'(r)\bigg)\\
&+\frac{1}{|\mathcal{P}_{n}|}\sum_{P\in \mathcal{P}_{n}}\psi_{1r}(P+h_1)\psi_{2r}(P+h_2)\big(\psi_{1r}^{*}(P+h_1)\psi_{2r}^{*}(P+h_2)-Q'(r, n)\big).
\end{align*}

\noindent 
Observe that $Q'(r, n)\ll 1$. Therefore we have
\begin{align}
&\frac{R_2(n, q)}{|\mathcal{P}_{n}|}-Q'(n)\ll \Big|\frac{1}{|\mathcal{P}_{n}|}\sum_{P\in \mathcal{P}_{n}}\psi_{1r}(P+h_1)\psi_{2r}(P+h_2)-Q'(r)\Big| \\
&+ \frac{1}{|\mathcal{P}_{n}|}\sum_{P\in \mathcal{P}_{n}}\big|\psi_{1r}^{*}(P+h_1)\psi_{2r}^{*}(P+h_2)-Q'(r, n)\big|=: E_9+E_{10}\nonumber
\end{align}
Using Lemma \ref{Chinese remainder theorem}, we have
\begin{align*}
&\sum_{P\in \mathcal{P}_{n}}\psi_{1r}(P+h_1)\psi_{2r}(P+h_2)=\sum_{P\in \mathcal{P}_{n}}\sum_{g_1|P+h_1}\alpha_{1r}(g_1)\sum_{g_2|P+h_2}\alpha_{2r}(g_2)\\
&=\sideset{}{'}\sum_{\substack{g_j\in \mathcal{M}_{\leq n}\\ j=1, 2}}\alpha_{1r}(g_1)\alpha_{2r}(g_2)\pi_{A}(n; [g_1, g_2], M)\\
&=\sideset{}{'}\sum_{\substack{g_j \in \mathcal{M}_{\leq z}\\ j=1, 2}}\alpha_{1r}(g_1)\alpha_{2r}(g_2)\Big(\pi_{A}(n; [g_1, g_2], M)
-\frac{q^n}{n\Phi([g_1, g_2])}\Big) + \frac{q^n}{n}\sideset{}{'}\sum_{\substack{g_j\in \mathcal{M}_{\leq z}\\ j=1, 2}}\frac{\alpha_{1r}(g_1)\alpha_{2r}(g_2)}{\Phi([g_1, g_2])}\\
&+O\Bigg(\sum_{\substack{g_1\in \mathcal{M}_{\leq n}\\ z<\deg (g_2)\leq n}}|\alpha_{1r}(g_1)\alpha_{2r}(g_2)|\pi_A(n; [g_1, g_2], M)\Bigg)
\end{align*}
where $M$ is the monic polynomial satisfying $M\equiv -h_j\pmod {g_j}, j=1, 2$ and 
$0\leq \deg (M)< \deg ([g_1, g_2])$, and $\sideset{}{'}\sum$ denotes summation over $g_1, g_2$ 
satisfying $(g_1, g_2)|(h_2-h_1)$, $\alpha_{jr}, j=1, 2$ 
are as defined in section $5$ and $r\leq z\le n/4$ is to be chosen later.

\noindent
Therefore we have
\begin{align*}
&E_9\leq \frac{1}{|\mathcal{P}_{n}|}\sideset{}{'}\sum_{\substack{g_j \in \mathcal{M}_{\leq z}\\ j=1, 2}}|\alpha_{1r}(g_1)\alpha_{2r}(g_2)|\left|\pi_{A}(n; [g_1, g_2], M)-\frac{q^n}{n\Phi([g_1, g_2])}\right|\\
&+
O\Bigg(\sum_{\substack{g_1\in \mathbb{F}_q[x]\setminus\{0\}\\ \deg (g_2)>z}}\frac{|\alpha_{1r}(g_1)\alpha_{2r}(g_2)|}{\Phi([g_1, g_2])}\Bigg)+O\Bigg(\frac{1}{|\mathcal{P}_{n}|}\sum_{\substack{g_1 \in \mathcal{M}_{\leq n}\\ z<\deg (g_2)\leq n}}|\alpha_{1r}(g_1)\alpha_{2r}(g_2)|\pi_A(n; [g_1, g_2], M)\Bigg)\\
&=:E_{11}+E_{12}+E_{13}.
\end{align*}
\noindent
If $\deg(g_1)$ and $\deg(g_2)$ are both $\le z\le n/4$ then $\deg([g_1,g_2])\le n/2$
and hence we can apply \eqref{primes in AP over function field} to get
\[
 \pi(n,[g_1,g_2], M)=\frac{q^n}{n\Phi([g_1,g_2])} +O\left(\frac{q^{\frac{n}{2}}}{n}\right).
\]
Using these estimates we obtain

\begin{align*}
E_{11}&\leq \frac{q^{n/2}}{n|\mathcal{P}_{n}|}
\sum_{g_1, g_2\in \mathcal{M}_{\le z}}
|\alpha_{1r}(g_1)\alpha_{2r}(g_2)|\\
&\ll \frac{q^{n/2+2z\alpha}}{n|\mathcal{P}_{n}|}
\sum_{g_1, g_2\in \mathcal{M}}
\frac{|\alpha_{1r}(g_1)\alpha_{2r}(g_2)|}{|g_1|^\alpha |g_2|^\alpha}\\
&\ll \frac{q^{n/2+2z\alpha}}{n|\mathcal{P}_{n}|}
\prod_{1\le j\le 2}
\left(\sum_{g_j\in \mathcal{M}}\frac{|\alpha_{jr}(g_j)|}{|g_j|^\alpha}\right)
\ll \frac{q^{n/2+2z\alpha}}{n|\mathcal{P}_{n}|}
\exp\Big(2c_3\frac{q^{r(1-\alpha)}}{r}\Big).
\end{align*}
Note that, trivially $\pi(n, [g_1,g_2], M) \ll \left(\frac{q^n}{|[g_1,g_2]|}+1\right)$.
Therefore
\begin{align*}
E_{13}&\ll \frac{1}{|\mathcal{P}_{n}|}
\sum_{\substack{g_1\in \mathcal{M}_{\leq n}\\ z<\deg (g_2)\leq n}}
|\alpha_{1r}(g_1)\alpha_{2r}(g_2)|\Big(\frac{q^n}{|[g_1, g_2]|}+1\Big)\\
&\ll \frac{q^{n-\alpha z}}{|\mathcal{P}_{n}|}r^{c_1}
\exp\left(c_4\frac{q^{\alpha r}}{r}\right)
+ \frac{q^{2n \alpha}}{|\mathcal{P}_{n}|}
\exp\Big(c_5\frac{q^{r(1-\alpha)}}{r}\Big)
\end{align*}
Similarly we have
\begin{align*}
E_{12}\ll q^{-\alpha z}r^{c_1}\exp\Big(c_6\frac{q^{\alpha r}}{r}\Big)\ll q^{-\alpha z}\exp\Big(c_7\frac{q^{\alpha r}}{r}\Big). 
\end{align*}
\noindent
Therefore  we write
\begin{align*}
E_{10}&= \frac{1}{|\mathcal{P}_{n}|}\sum_{P\in \mathcal{Q}_{r}}\big|\psi_{1r}^{*}(P+h_1)\psi_{2r}^{*}(P+h_2)-Q'(r, n)\big|\\
&+ \frac{1}{|\mathcal{P}_{n}|}\sum_{P\not\in \mathcal{Q}_{r}}\big|\psi_{1r}^{*}(P+h_1)\psi_{2r}^{*}(P+h_2)-Q'(r, n)\big|=:\frac{1}{|\mathcal{P}_{n}|}(E_{14}+E_{15}).
\end{align*}
Using Lemma \ref{cardinality-N_r-Q_r}, we get
\[
 E_{14}\ll |\mathcal{Q}_r|
 \ll |\mathcal{P}_{n}|\sum_{j=1}^{2}\mathbb{D}(\psi_j, 1; r, n)
  +\frac{|\mathcal{P}_{n}|}{rq^r}+\frac{|\mathcal{P}_{n}|} {q^{\frac{n}{4}}}.
\]
\noindent
Using \eqref{exp} and \eqref{logarithm}, we have
\begin{align*}
\left|\psi_{1r}^{*}(P+h_1)\psi_{2r}^{*}(P+h_2)-Q'(r, n)\right|&\leq \bigg|\sum_{j=1}^{2}\sum_{\substack{Q^k\| P+h_j\\ \deg Q>r}}\left(\psi_j(Q^k)-1\right)-\log Q'(r, n)\bigg|\\
&+ O\bigg(\sum_{j=1}^{2}\sum_{\substack{Q^k\| P+h_j\\ \deg Q>r}}\left|\psi_j(Q^k)-1\right|^2\bigg). 
\end{align*}
Therefore we get
\begin{align*}
 & E_{15}\leq \sum_{P\not\in \mathcal{Q}_{r}}
\Bigg|
\sum_{j=1}^2\sum_{\substack{Q^{k}\| P+h_j\\ \deg Q>r}}
(\psi_j(Q^{k})-1)-\log Q'(r,n)
\Bigg|+O\Bigg(\sum_{P\not\in \mathcal{Q}_{r}}\sum_{j=1}^2
\sum_{\substack{Q^k\| P+h_j\\ \deg Q>r}}\left|\psi_j(Q^k)-1\right|^2\Bigg)
\\
& \le \sum_{P\not\in \mathcal{Q}_{r}}
\Bigg|
\sum_{j=1}^2\sum_{\substack{Q^{k}\| P+h_j\\ \deg Q>r}}
(\psi_j(Q^{k})-1)-\sum_{j=1}^2\sum_{\substack{k\ge 1 \\ r<\deg(Q)\le n/k}}
\frac{\psi_j(Q^k)-1}{|Q^k|}
\Bigg|\\
&+\sum_{P\not\in \mathcal{Q}_{r}}
\Bigg|\sum_{j=1}^2\sum_{\substack{k\ge 1 \\ r<\deg(Q)\le n/k}}
\frac{\psi_j(Q^k)-1}{|Q^k|}
-\log Q'(r,n)
\Bigg|
+O\Bigg(\sum_{P\not\in \mathcal{Q}_{r}}\sum_{j=1}^2
\sum_{\substack{Q^k\| P+h_j\\ \deg Q>r}}\left|\psi_j(Q^k)-1\right|^2\Bigg)\\
& = E_{18}+E_{19}+E_{20}.
\end{align*}
Applying Lemma \ref{useful turan-kubilius over primes} with $a(Q^k)=\psi_j(Q^k)-1$ for 
$Q\in \mathcal{P}$ and $\deg(Q)>r$, we get
\[
E_{18}\ll |\mathcal{P}_{n}|\bigg(\sum_{\substack{k\deg Q\leq n\\k\geq 1; \deg Q>r}}
\frac{|\psi_j(Q^k)-1|^2}{\Phi(Q^k)}\bigg)^{\frac{1}{2}}\ll |\mathcal{P}_{n}| 
\sum_{j=1}^{2}\mathbb{D}(\psi_j, 1; r, n)
+\frac{|\mathcal{P}_{n}|}{(rq^r)^{\frac{1}{2}}}.
\]
Observe that
\begin{align}
\label{definition of P(r, n)} 
Q'(r, n)=\prod_{r<\deg P\leq n}\bigg(1-\frac{2}{\Phi(P)}+\sum_{k=1}^{\infty}\frac{\psi_1(P^k)+\psi_2(P^k)}{q^{k\deg P}}\bigg).
\end{align}
It is easy to see that
\[
 \log Q'(r,n)=\sum_{\substack{Q\in \mathcal{P} \\ r<\deg(Q)\le n}}
 \left(-\frac{2}{\phi(Q)}+\sum_{j=1}^2 \sum_{m=1}^{\infty}\frac{\psi_j(Q^m)}{|Q^m|}\right)
 +O\left(\frac{1}{rq^r}\right)
\]
and consequently
\[
E_{19}\ll \frac{|\mathcal{P}_{n}|}{rq^r}.
\]
In order to estimate $E_{20}$, we first consider the sum corresponding to $k=1$ and a fixed 
$j\in \{1,2\}$ which is as follows.
\begin{align*}
& \sum_{P\not\in \mathcal{Q}_r}\sum_{\substack{Q\| P+h_j \\ \deg(Q)>r}}|1-\psi_j(Q)|^2 \\
&\le \sum_{\substack{Q\in\mathcal{P} \\ r<\deg(Q)\le n}}|1-\psi_j(Q)|^2 \pi(n, Q, -h_j)\\
&\le \sum_{\substack{Q\in\mathcal{P} \\ r<\deg(Q)\le n/2}}|1-\psi_j(Q)|^2 \pi(n, Q, -h_j)
+\sum_{\substack{Q\in\mathcal{P} \\ n/2<\deg(Q)\le n}}
\frac{|1-\psi_j(Q)|}{\phi(Q)^{1/2}}\phi(Q)^{1/2} \pi(n, Q, -h_j)\\
&\ll \frac{q^n}{n}\sum_{\substack{Q\in\mathcal{P} \\ r<\deg(Q)\le n/2}}
\frac{|1-\psi_j(Q)|^2}{\phi(Q)}
+\Bigg( \sum_{\substack{Q\in\mathcal{P} \\ n/2<\deg(Q)\le n}}
\frac{|1-\psi_j(Q)|^2}{\phi(Q)}\Bigg)^{1/2}(\Theta(n))^{1/2}
\end{align*}

Similar to estimation of $E_{14}$, we get
\[
E_{20}\ll |\mathcal{P}_{n}| \sum_{j=1}^{2}\mathbb{D}(\psi_j, 1; r, n)+\frac{|\mathcal{P}_{n}|}{rq^r}+\frac{|\mathcal{P}_{n}|}{q^{\frac{n}{4}}}. 
\]
\noindent
Choosing $z=A\log_q n$, $A>0$ and combining all these estimates, we get the required theorem.

\section{Proof of Theorem \ref{theorem for hayes pretentious} and Corollary \ref{corollary on different functions}}
   \subsubsection*{Proof of Theorem \ref{theorem for hayes pretentious}}
   Given $f\in \mathcal{M}_n$, define $f_1, f_2\in \mathbb{F}_q[x]\setminus \{0\}$ such that 
    \[
    f_1|(f+h_1)\, \,  \text{ and } \,\,\rad(f_1)|Q_1, \, \left(\frac{f+h_1}{f_1}, Q_1\right)=1
    \] and 
    \[
    f_2|(f+h_2) \,\,\text{ and } \,\, \rad(f_2)|Q_2,\, \left(\frac{f+h_2}{f_2}, Q_2\right)=1.
    \] 
    Note that $(f_1, f_2)|(h_1-h_2).$ So, we can write
    \[
    f+h_1=G\frac{f_1f_2}{(f_1, f_2)}+D_1f_1 \text{ and } f+h_2=G\frac{f_1f_2}{(f_1, f_2)}+D_2 f_2 
    \]
    such that $D_1f_1-D_2f_2=h_1-h_2$, where $D_1$ and $D_2$ are polynomials that depends on $f_1$ and $f_2$.
    
    \vspace{2mm}
    \noindent   
Therefore, using multiplicativity of  $\psi_j$'s, we obtain
    \begin{align*}
    &\sum_{f\in \mathcal{M}_n}\psi_1(f+h_1)\psi_2(f+h_2)\\
    &=\sum_{\substack{\rad(f_1)| Q_1\\ \rad(f_2)| Q_2}}\psi_1(f_1)\psi_2(f_2)\sideset{}{^*}\sum_{d(G)=n-d([f_1, f_2])}\psi_1\left(G\frac{f_2}{(f_1, f_2)}+ D_1\right)\psi_2\left(G\frac{f_1}{(f_1, f_2)}+ D_2\right),
    \end{align*}
    where the sum $\sum^{*}$ varies over the polynomials $G$ such that 
    \begin{align}\label{1st condition on sum}
    \left(G\frac{f_2}{(f_1, f_2)}+ D_1, Q_1\right)=1, \quad \left(G\frac{f_1}{(f_1, f_2)}+ D_2, Q_2\right)=1,
    \end{align}
  
 \noindent
   Define multiplicative functions $\widetilde{\psi}_j, j=1, 2$ by
    \[
     \widetilde{\psi}_j(P^k)= \left\{
	\begin{array}
	[c]{ll}
	\psi_j(P^k) & \text{if}\; \, P\nmid Q_j,\\
	0 & \text{otherwise}.
	\end{array}
	\right.
    \]
    This gives us 
    \begin{align*}
    & \sideset{}{^*}\sum_{d(G)=n-d([f_1, f_2])}\psi_1\left(G\frac{f_2}{(f_1, f_2)}+ D_1\right)\psi_2\left(G\frac{f_1}{(f_1, f_2)}+D_2\right)\\
    &=\sum_{d(G)=n-d([f_1, f_2])}\widetilde{\psi}_1\left(G\frac{f_2}{(f_1, f_2)}+ D_1\right)\widetilde{\psi}_2\left(G\frac{f_1}{(f_1, f_2)}+ D_2\right).
    \end{align*}
    Now we write $\widetilde{\psi}_j$ as 
    \[
    \widetilde{\psi}_j(f)=\chi_j(f) \psi'_j(f), \quad j=1, 2,
    \]
    where $\chi_j$ are the Dirichlet characters in the hypothesis. Then $\chi_1 \chi_2$ is a Dirichlet character modulo $[Q_1, Q_2]$.
    
    \vspace{2mm}
    \noindent
In the above sum we write  
  \[
    G=g[Q_1, Q_2]+h,
    \]
     where $h$ runs over residue classes modulo $[Q_1, Q_2]$. From the Hypothesis, we have that $d(f+h)=d(f)$ for sufficiently large degree of $f$.  Therefore the above sum becomes
    \begin{align*}
    &e^{2\pi i (\theta_1+\theta_2)n}e_{\theta_1}\left(\frac{f_2}{(f_1, f_2)}\right)e_{\theta_2}\left(\frac{f_1}{(f_1, f_2)}\right)e_{-(\theta_1+\theta_2)}([f_1, f_2])\\
    &\times \sum_{h([Q_1, Q_2])}\chi_1\left(\frac{hf_2}{(f_1, f_2)}+ D_1\right)\chi_2\left(\frac{hf_1}{(f_1, f_2)}+D_2\right)\\
  &  \times \sum_{\substack{d(g)=n-d([f_1, f_2])\\-d([Q_1, Q_2])}}(\psi'_1e_{-\theta_1})\left(\frac{gf_2[Q_1, Q_2]}{(f_1, f_2)}+\frac{hf_2}{(f_1, f_2)} + D_1\right)(\psi'_2e_{-\theta_2})\left(\frac{gf_1[Q_1, Q_2]}{(f_1, f_2)}+\frac{hf_1}{(f_1, f_2)} + D_2\right)
    \end{align*}
 
   \noindent
 We apply Theorem \ref{main theorem 1} to the innermost sum with the condition that
    \begin{align}\label{det. cond.}
    \Delta=\frac{[Q_1, Q_2]}{(f_1, f_2)}(h_1-h_2).
    \end{align}
    Since the inner sum does not depend on the residue classes modulo $[Q_1, Q_2]$, so upto a small error of $o(1)$, it is equal to 
    \begin{align*}
    \frac{1}{|[Q_1, Q_2]|}\sum_{d(G)=n-d([f_1, f_2])}(\psi'_1 e_{-\theta_1})\left(G\frac{f_2}{(f_1, f_2)}+ D_1\right)(\psi'_2e_{-\theta_2})\left(G\frac{f_1}{(f_1, f_2)}+ D_2\right).
    \end{align*}

    \vspace{2mm}
    \noindent
    Gathering these estimates, we conclude that
    \begin{align*}
    &\sum_{f\in \mathcal{M}_n}\psi_1(f+h_1)\psi_2(f+h_2)=e^{2\pi i (\theta_1+\theta_2)n}\frac{1}{|[Q_1, Q_2]|}\\&\times \sum_{\substack{\rad(f_1)| Q_1\\ \rad(f_2)| Q_2}}\psi_1(f_1)\psi_2(f_2)e_{\theta_1}\left(\frac{f_2}{(f_1, f_2)}\right)e_{\theta_2}\left(\frac{f_1}{(f_1, f_2)}\right)e_{-(\theta_1+\theta_2)}([f_1, f_2])\\
    &\times \sum_{h([Q_1, Q_2])}\chi_1\left(\frac{hf_2}{(f_1, f_2)}+D_1\right)\chi_2\left(\frac{hf_1}{(f_1, f_2)}+ D_2\right)\\
 &   \times \sum_{d(G)=n-d([f_1, f_2])}(\psi'_1e_{-\theta_1})\left(G\frac{f_2}{(f_1, f_2)}+ D_1\right)(\psi'_2e_{-\theta_2})\left(G\frac{f_1}{(f_1, f_2)}+ D_2\right).
    \end{align*}

    \noindent
   Using Lemma \ref{Single character sum} and Lemma \ref{Double character sum}, the character sum
   \begin{align}\label{DCS}
   \sum_{h([Q_1, Q_2])}\chi_1\left(\frac{hf_2}{(f_1, f_2)}+D_1\right)\chi_2\left(\frac{hf_1}{(f_1, f_2)}+ D_2\right)
   \end{align}
   vanishes unless $\frac{Q_2}{(Q_1, Q_2)}\mid  \frac{f_1}{(f_1, f_2)}$ and $\frac{Q_1}{(Q_1, Q_2)}\mid  \frac{f_2}{(f_1, f_2)}$. 
   
   \vspace{2mm}
   \noindent
Observe that the hypothesis $\mathbb{D}(\psi_j, \chi_j e_{\theta_j} ; \infty)<\infty$ implies that \[
    \mathbb{D}(\psi'_je_{-\theta_j}, 1; \infty)<\infty.
    \]
We use Theorem \ref{main theorem 1} to the above innermost sum 
   to conclude the proof.
   
   \subsubsection*{Proof of Corollary \ref{corollary on different functions}} In this case, we have $(f_1, f_2)=1$. From \eqref{DCS}, \begin{align*}
   \sum_{h([Q_1, Q_2])}\chi_1\left(hf_2+D_1\right)\chi_2\left(hf_1+ D_2\right)
   \end{align*}
   vanishes unless 
   \[
  U= f_1=\frac{Q_2}{(Q_1, Q_2)}\,  \text{ and } \, V=f_2=\frac{Q_1}{(Q_1, Q_2)}.
   \] 

\section{Proof of Theorem \ref{main theorem on hayes pretentiousness}}
  Recall that 
  \begin{align*}
 \mathcal{M}_n=\bigsqcup_{B\in \mathcal{B}}I(B; n-l),
 \end{align*}
 where 
 \[
 \mathcal{B}=\{B=t^n+b_{n-1}t^{n-1}+\ldots + b_{n-l}t^{n-l}: \, b_j\in \mathbb{F}_q\}.
 \]
 So, we have
   \[
    \sum_{f\in \mathcal{M}_n}\psi(f+h_1)\overline{\psi(f+h_2)}=\sum_{B\in \mathcal{B}}\sum_{f\in I(B; n-l)}\psi(f+h_1)\overline{\psi(f+h_2)}.
   \]
   We see that if $f\in I(B; n-l)$ then we have $f+h_j\in I(B; n-l)$, since $\deg(h_j)\leq l$ for all $j=1, 2$. Therefore 
   \[
   f+h_j\in I(B; n-l)\Longrightarrow \xi(f+h_j)=\xi(B).
   \]
   This gives us
   \begin{align*}
   \sum_{B\in \mathcal{B}}\sum_{f\in I(B; n-l)}\psi(f+h_1)\overline{\psi(f+h_2)}&=\sum_{B\in \mathcal{B}}\xi(B)\overline{\xi(B)}\sum_{f\in I(B; n-l)}(\psi \overline{\xi})(f+h_1)\overline{(\psi \overline{\xi})(f+h_2)}\\
   &=\sum_{B\in \mathcal{B}}\sum_{f\in I(B; n-l)}(\psi \overline{\xi})(f+h_1)\overline{(\psi \overline{\xi})(f+h_2)}\\
   &=\sum_{f\in \mathcal{M}_n}(\psi \overline{\xi})(f+h_1)\overline{(\psi \overline{\xi})(f+h_2)}.
   \end{align*}
 The hypothesis $
   \mathbb{D}(\psi, \chi \xi e_{\theta}; \infty)<\infty,
   $ implies that $\mathbb{D}(\psi \overline{\xi}, \chi e_{\theta}; \infty)<\infty$. Hence we can apply Theorem \ref{theorem on Dirichlet character} to conclude the proof.  
  
\section{Proof of Theorem \ref{theorem on Dirichlet character} and Corollary \ref{main theorem on hayes character}}

\subsubsection*{Proof of Theorem \ref{theorem on Dirichlet character}}
 We follow the arguments of proof of the Theorem \ref{theorem for hayes pretentious} to obtain
  \begin{align*}
    &\sum_{f\in \mathcal{M}_n}\psi(f+h_1)\overline{\psi(f+h_2)}=\frac{1}{|Q|}\sum_{\substack{\rad(f_1)| Q\\ \rad(f_2)| Q\\ (f_1, f_2)\mid \, \Delta}}\psi(f_1)\overline{\psi(f_2)}e_{-\theta}\left(\frac{f_2}{(f_1, f_2)}\right)e_{\theta}\left(\frac{f_1}{(f_1, f_2)}\right)\\
    &\times \sum_{h(Q)}\chi\left(h\frac{f_2}{(f_1, f_2)}+D_1\right)\overline{\chi\left(h\frac{f_1}{(f_1, f_2)}+ D_2\right)}\\
    &\times \sum_{d(G)=n-d([f_1, f_2])}(\psi \overline{\chi}e_{-\theta})\left(G\frac{f_2}{(f_1, f_2)}+ D_1\right)(\overline{\psi}\chi e_{\theta})\left(G\frac{f_1}{(f_1, f_2)}+ D_2\right),
    \end{align*} 
    where  $D_1, D_2$ are polynomials depending on $f_1$ and $f_2$ such that $D_2f_2-D_1f_1=\Delta$.\\
    Now we have to estimate 
    \[
    T(Q):= \sum_{h(Q)}\chi\left(h\frac{f_2}{(f_1, f_2)}+D_1\right)\overline{\chi\left(h\frac{f_1}{(f_1, f_2)}+ D_2\right)}.
    \]
    By Chinese remainder theorem on $\mathbb{F}_q[x]$, we have
    \[
    T(Q)=\prod_{P^k\parallel Q}\sum_{h(P^k)}\chi_{P^k}(hf_2+D_1)\overline{\chi_{P^k}(hf_1+D_2)},
    \]
    where $\chi_{P^k}$ is a primitive Dirichlet character of conductor $P^k$.
    
    \vspace{2mm}
    \noindent
    We claim that $T(Q)$ vanishes when $f_1\neq f_2$. In this case, there exists a irreducible polynomial say $P\in \mathbb{F}_q[x]$ such that $Pi^\Vert f_1$ and $P^j\vert f_2$ with $j>i$. Then $\left(\frac{f_1}{(f_1, f_2)}, P\right)=1$. Therefore using the change of variable 
    \[
    h \longmapsto h\frac{f_1}{(f_1, f_2)} \pmod{P^k}
    \]
    the inner sum of $T(Q)$ becomes
    \[
    S(P)=\sum_{h(P^k)}\chi_{P^k}(hP^{j-i}t+D_1')\overline{\chi_{P^k}}(h+D_2'),
    \]
    where $(t, P)=1$ and for polynomials $D_j'\in \mathbb{F}_q[x]$, $j=1, 2$. If $j-i\geq k$ then first term of the sum $S(P)$ is fixed and the second term runs over all residue classes modulo $P^k$, which leads the sum to be zero.
    
    \vspace{2mm}
    \noindent
    Let us assume that $j-i<k$.  We write 
    \[
    h=H+P^{k-(j-i)}F,
    \]
    where $H$ runs over all residue classes modulo $P^{k-(j-i)}$ and $L$ over residue classes modulo $P^{j-i}$. Thus, we obtain
    \[
    S(P)=\sum_{H(P^{k-(j-i)})}\chi_{P^k}(HP^{j-i}t+D_1')\sum_{L(P^{j-i})}\overline{\chi_{P^k}(H+P^{k-(j-i)}F+D_2')}.
    \]
   Applying Lemma \ref{Basic lemma}, we say that the inner sum of $S(P)$ vanishes. Hence, $f_1=f_2=f$ (say). Therefore, we conclude that 
   \begin{align*}
   &\sum_{f\in \mathcal{M}_n}\psi(f+h_1)\overline{\psi(f+h_2)}=\frac{1}{|Q|}\sum_{\substack{\rad(f)|\, Q\\ f\mid \, \Delta}}|\psi(f)|^2\\
    &\times \sum_{h(Q)}\chi\left(h+D_1\right)\overline{\chi\left(h+ D_2\right)}\\
    &\times \sum_{d(G)=n-d(f)}(\psi \overline{\chi} e_{-\theta})\left(G+ D_1\right)(\overline{\psi}\chi e_{\theta})\left(G+ D_2\right),
    \end{align*} 
    where  $D_1, D_2$ are the polynomials depend on $f_1$ and $f_2$ such that $D_2f_2-D_1f_1=\Delta$.\\
   We use Theorem \ref{main theorem 1} to the innermost sum to conclude the proof.

\subsubsection*{Proof of Corollary \ref{main theorem on hayes character}} We apply $h_1=0$ and $h_2=1$ to the Theorem \ref{main theorem on hayes pretentiousness}. The conditions $f\mid \Delta$ and $(D_1-D_2)f=\Delta$ implies that 
$f=1$ and $D_2-D_1=1$. Also $\deg(D_j)\leq 0$ forces $D_2=1$ and $D_1=0$. Therefore,
    \begin{align*}
    S(P)=\sum_{h(P^k)}\chi_{P^k}(h)\overline{\chi_{P^k}}(h+1)=\left\{
	\begin{array}
	[c]{ll}
	 -1 & \text{ if }\;  k=1\\
 0 & \text{ if }\,  k\geq 2.
	\end{array}
	\right.
    \end{align*}
    This yields $T(Q)=\mu(Q)$, which concludes the proof.

  \section{Proof of Theorem \ref{Katai conjecture} and Corollary \ref{Katai conjectiure for pair}}
 \subsubsection*{Proof of Theorem \ref{Katai conjecture}}
  Let us consider
     \[
     \Delta \psi(f)=\psi(f+1)-\psi(f).
     \]
     Using Hypothesis we observe that  
     \begin{align}\label{main observation}
     \sum_{f\in \mathcal{M}_{\leq N}}\frac{|\Delta \psi(f)|^2}{|f|}\leq \sum_{f\in \mathcal{M}_{\leq N}}\frac{2|\Delta \psi(f)|}{|f|}= 2\sum_{n\leq N}\frac{1}{q^n}\sum_{f\in \mathcal{M}_n}|\Delta \psi(f)|=o(N).
     \end{align}
     \noindent
     \textbf{Step $1$:} We first show that if for some $0<\epsilon <1$,
     \begin{align}\label{mean square bound}
     \sum_{f\in \mathcal{M}_{\leq N}}\frac{|\Delta \psi(f)|^2}{|f|}\leq 2(1-\epsilon)N
     \end{align}
     holds then there exists a primitive character $\chi$, a short interval character $\xi$ and an angle $\theta\in [0, 1]$ such that 
     \[
     \mathbb{D}\left(\psi(f), \chi(f)\xi(f) e_{\theta}(f); \infty\right)< \infty.
     \]
     To prove the claim we start by writing
     \[
     \Re (\psi(f)\overline{\psi(f+1)})=1-\frac{|\Delta \psi(f)|^2}{2}
     \]
     so that \eqref{mean square bound} gives
     \[
     \sum_{f\in \mathcal{M}_{\leq N}}\frac{\Re (\psi(f)\overline{\psi(f+1)})}{|f|}\geq \epsilon N.
     \]
     We can apply the Theorem $1.5$ of \cite{KMT} to deduce that for every sufficiently large $N$, there exist a Dirichlet character $\chi_N$ of bounded modulus, a short interval character $\xi_N$ of bounded length, and an angle $\theta_N\in [0, 1)$ such that 
  \[
  \mathbb{D}(\psi, \chi_N \xi_N e_{\theta_N}; N)\ll 1.
  \]
  Following the argument in page $54$ of \cite{KMT}, we can conclude that uniformly in $N$,
  \[
  \mathbb{D}(\psi, \chi \xi e_{\theta}; N)\ll 1,
   \]
     which establishes the claim.
     
     \vspace{2mm}
     \noindent
     \textbf{Step $2$:} We now show that if $\mathbb{D}\left(\psi(f), \chi(f)\xi(f) e_{\theta}(f); \infty\right)< \infty$ for a primitive Dirichlet character $\chi$ of modulus $Q\in \mathbb{F}_q[x]$, a short interval character of length $l\geq 1$ and an angle $\theta\in [0, 1]$ then
     \[
     \sum_{f\in \mathcal{M}_{\leq N}}\frac{|\Delta \psi(f)|^2}{|f|}=2\left(1- \mathbb{E}(\psi)+o(1)\right)N, 
     \]
     where 
     \[
     \mathbb{E}(\psi)=\frac{\mu(Q)}{|Q|}\prod_{P\nmid Q}\bigg(2\left(1-\frac{1}{|P|}\right)\bigg(\sum_{m=0}^{\infty}\frac{\Re\left((\psi\chi\xi e_{\theta})(P^m)\right)}{|P|^m}\bigg)-1\bigg).
     \]
     This step easily follows from Corollary \ref{main theorem on hayes character} together with the estimate that
     \[
      \sum_{f\in \mathcal{M}_{\leq N}}\frac{|\Delta \psi(f)|^2}{|f|}=2\bigg(\sum_{f\in \mathcal{M}_{\leq N}}\frac{1-\Re (\psi(f)\overline{\psi(f+1)})}{|f|}\bigg).
     \]
     \noindent
     \textbf{Step $3$:} Combining Step $1$, Step $2$, and observation \eqref{main observation}, we must have that
     $\mathbb{E}(\psi)=1$. From this condition we have to find out the the desired form of the  function $\psi$.\\
     Observe that the euler factor 
     \[
     2\left(1-\frac{1}{|P|}\right)\bigg(\sum_{m=0}^{\infty}\frac{\Re\left((\psi\overline{\chi}\overline{\xi} e_{-\theta})(P^m)\right)}{|P|^m}\bigg)-1\geq \frac{|P|-4}{|P|}\geq -1,
     \]
     where the equality holds only when $|P|=2$ $\left(\deg(P)=\lfloor \frac{\log 2}{\log q}\rfloor\right)$ and also
     \[
      2\left(1-\frac{1}{|P|}\right)\bigg(\sum_{m=0}^{\infty}\frac{\Re\left((\psi\overline{\chi}\overline{\xi} e_{-\theta})(P^m)\right)}{|P|^m}\bigg)-1\leq 2\left(1-\frac{1}{|P|}\right)\left(\sum_{m=0}^{\infty}\frac{1}{|P|^m}\right)-1\leq 1. 
     \]
     Therefore we must have $Q\in \mathbb{F}_q^{*}$ and for all $P\in \mathcal{P}$, 
     \[
     2\left(1-\frac{1}{|P|}\right)\bigg(\sum_{m=0}^{\infty}\frac{\Re\left((\psi\overline{\xi} e_{-\theta})(P^m)\right)}{|P|^m}\bigg)-1=1,
     \]
     which is possible if and only if $\psi(P^m)=\xi(P^m)e_{\theta}(P^m)$ for all $m\geq 1$.

    \subsubsection*{Proof of Corollary \ref{Katai conjectiure for pair}}
     We can write the hypothesis as
    \begin{align}\label{difference between two functions on average}
    \sum_{f\in \mathcal{M}_n}\frac{|\psi(f+1)-\eta(f)|}{|f|}\to 0, \quad \text{ as } n\to \infty.
    \end{align}
    For any $A\in \mathcal{M}$, we consider $h=A(f+1)-1$. So \eqref{difference between two functions on average} implies that
    \begin{align}\label{shifting difference on average}
     \sum_{h\in \mathcal{M}_{n+\deg(A)}}\frac{|\psi(h+1)-\eta(h)|}{|h|}\to 0, \quad \text{ as } n\to \infty.
    \end{align}
    On the other hand, for $A\in \mathcal{M}$, we also have
    \begin{align*}
    \sum_{f\in \mathcal{M}_n}\frac{\bigg|\frac{\psi(A(f+1))}{\psi(A)}-\eta(f)\bigg|}{|f|}\to 0, \quad \text{ as } n\to \infty,
    \end{align*}
    which turns into
     \begin{align}\label{shifting difference on average 2}
    \sum_{f\in \mathcal{M}_n}\frac{|\psi(A(f+1))-\psi(A)\eta(f)|}{|f|}\to 0, \quad \text{ as } n\to \infty.
    \end{align}
   Also we write \eqref{shifting difference on average} as
    \begin{align}\label{shifting difference on average 3}
     \sum_{f\in \mathcal{M}_n}\frac{|\psi(A(f+1))-\eta(A(f+1)-1)|}{|Af|}\to 0,\quad \text{ as } n\to \infty.
    \end{align}
    From \eqref{shifting difference on average 2} and \eqref{shifting difference on average 3}, we obtain
    \begin{align}\label{shifting difference on average 4}
    \sum_{f\in \mathcal{M}_n}\frac{|\psi(A)\eta(f)-\eta(A(f+1)-1)|}{|f|}\to 0,\quad \text{ as } n\to \infty.
    \end{align}
    We use change of variable $f=(A-1)Bg$ for some $B\in \mathcal{M}$ so that $\deg(g)=\deg(f)-\deg(A)-\deg(B)$. Let $k=n-\deg(A)-\deg(B)$.
    Therefore, \eqref{shifting difference on average 4} becomes
    \begin{align*}
    \sum_{f\in \mathcal{M}_k}\frac{|\psi(A)\eta(A-1)\eta(B)\eta(g)-\eta(A-1)\eta(ABg+1)|}{|g|}\to 0,\quad \text{ as } k\to \infty,
    \end{align*}
    which implies that 
     \begin{align}\label{shifting difference on average 5}
    \sum_{f\in \mathcal{M}_k}\frac{|\psi(A)\eta(B)\eta(g)-\eta(ABg+1)|}{|g|}\to 0,\quad \text{ as } k\to \infty.
    \end{align}
    From the symmetry in the $A$ and $B$, we also have
     \begin{align}\label{shifting difference on average 6}
    \sum_{f\in \mathcal{M}_k}\frac{|\psi(B)\eta(A)\eta(g)-\eta(ABg+1)|}{|g|}\to 0,\quad \text{ as } k\to \infty.
    \end{align}
    Therefore, \eqref{shifting difference on average 5} and \eqref{shifting difference on average 6} give us 
    \[
    \psi(A)\eta(B)=\psi(B)\eta(A).
    \]
    The function $H: \mathcal{M}\to \mathbb{S}^1$ defined by 
    \[
    H=\frac{\psi}{\eta}
    \]
     such that $H(A)=H(B)$ for all $A, B\in \mathcal{M}$. This leads us to conclude that $H$ is constant on $\mathcal{M}$. Then the complete multiplicativity of $H$ implies that $H=1$, which gives $\psi=\eta$. 
     The rest the proof follows from Theorem \ref{Katai conjecture}.
     
     \section{Proof of Corollary \ref{trancated liouville function}}
\noindent
We choose $r=y$ and $\psi_j=\lambda_y$, $j=1, 2$. Let $\alpha_j=\mu*\lambda_y$, $j=1, 2$.
On the basis of this choice, we find that $\mathbb{D}\left(\lambda_y, 1; r, n\right)=0$ and 
\[
\alpha_{j}(P^{t})=
\begin{cases}
2(-1)^{t}  \ & \text{if } \deg P\leq y \\
0 \ & \text{if } \deg P>y .
\end{cases} 
\]
We use Theorem \ref{main theorem 1} to obtain
\begin{align*}
\frac{1}{q^n}\bigg|\sum_{f\in \mathcal{M}_{n}}\lambda_y(f)\lambda_y(f+h)\bigg|\leq |Q(n)|+ O\bigg((yq^y)^{-\frac{1}{2}}+q^{(1-2\alpha)n}\exp\left(\frac{cq^{\alpha y}}{y}\right)\bigg),
\end{align*}
where $Q(n)$ is defined by \eqref{main term over monics introduction}.
Since $\deg (h)\leq y$ then we have
\[
Q(n)=\prod_{\deg P\leq y}\underset{\substack{\left(P^{m_1},P^{m_2}\right) |h}}{\sum_{\substack{m_{1}=0 }} ^{\infty}\sum_{\substack{m_{2}=0}}^{\infty}}\frac{\alpha_1(P^{m_1})\alpha_2(P^{m_2})}{q^{[m_1, m_2]\deg P}}=\prod_{\deg P\leq y}v_P.
\]
\noindent
Note that $\alpha_1=\alpha_2=\alpha_3$ (say). We define the non-negative integer $k(P)$ such that $P^{k(P)}\|h$.
For $\deg P\leq y$, we get
\begin{align*}
v_P&=\sum_{m=0}^{k(P)}\frac{\alpha_3(P^m)^2}{q^{m\deg P}}+2\sum_{m=0}^{k(P)}\alpha_3(P^m)\sum_{l=m+1}^{\infty}\frac{\alpha_3(P^l)}{q^{l\deg P}}\bigg(1+4\sum_{m=0}^{k(P)}\frac{1}{q^{m\deg P}}\bigg)+4\sum_{l=1}^{\infty}\frac{(-1)^l}{q^{l\deg P}}\\
&+ 8\sum_{m=1}^{k(P)}\frac{(-1)^{2m+1}}{q^{(m+1)\deg P}}\sum_{j=0}^{\infty}\frac{(-1)^j}{q^{j\deg P}}=1-\frac{4}{q^{k(P)\deg P}\left(q^{\deg P}+1\right)}.
\end{align*}
Finally, using Lemma \ref{some useful sums} and Lemma \ref{main lemma for TLF} and the hypothesis that $\deg (h) \leq y$, we have
\begin{align*}
Q(n)&= \prod_{\deg P\leq y}\bigg(1-\frac{4}{q^{\deg P}+1}\bigg)\prod_{\substack{\deg P\leq y\\ P|h}}\bigg(1-\frac{4}{q^{k(P)\deg P}\left(q^{\deg P}+1\right)}\bigg)\bigg(1-\frac{4}{q^{\deg P}+1}\bigg)^{-1}\\
&\leq C_1 \exp \bigg(-4\sum_{\deg P\leq y}q^{-\deg P}+4\sum_{\substack{\deg P\leq y\\ P\|h}}q^{-\deg P}\bigg)\leq C\frac{(\log y)^4}{y^4}.
\end{align*}
Using the Hypothesis that $2\leq y\leq \log n$, we have
\[
q^{(1-2\alpha)n}\exp\left(\frac{cq^{\alpha y}}{y}\right)\ll (y\log y)^{-1}.
\] 
Combining the above estimates we conclude the proof.
  
\section{Proof of Theorem \ref{distribution of additive function over function field introduction}}
\noindent
In the case of monic polynomials, the distribution function is 
\[
F_n(x):=\frac{1}{q^n} \nu_n\big\{f; \psi_1(f+h_1)+\psi_2(f+h_2)\leq x\big\}
\]
 and the corresponding characteristic function is 
 \[
 \phi_n(t)=\frac{1}{q^n}\sum_{f\in \mathcal{M}_{n}}\exp \left(it\left(\psi_1(f+h_1)+\psi_2(f+h_2)\right)\right).
 \]

\noindent
We observe that 
\begin{align*}
\sum_{P}\frac{\exp(it\psi_j(P))-1}{q^{\deg P}}=t\sum_{|\psi_j(P)|\leq 1}\frac{\psi_j(P)}{q^{\deg P}}+O\bigg(t^2 \sum_{|\psi_j(P)|\leq 1}\frac{\psi_j^2(P)}{q^{\deg P}}+\sum_{|\psi_j(P)|>1}q^{-\deg P}\bigg).
\end{align*}
Therefore from the hypothesis of the theorem, we can say that $\phi(t)$ is convergent for every real $t$. Further, the infinite product $\phi(t)$ is continuous at $t=0$ because it converges uniformly for $|t|\le T$ where $T>0$ is arbitrary.

\noindent
Also, for $j=1, 2$, we have
\begin{align*}
\mathbb{D}(\psi_j(P), 1;  \infty)\ll t^2 \sum_{|\psi_j)(P)|\leq 1}\frac{\psi_j^2(P)}{q^{\deg P}}+\sum_{|\psi_j(P)|>1}q^{-\deg P}.
\end{align*}

\vspace{2mm}
\noindent
So, using the hypothesis of the theorem we see that $\psi_j$ is close to $1$ and choosing $r=\log n$ in Theorem \ref{main theorem 1} we get that the remainder term disappears when $n\to \infty.$
 
 \noindent
Thus the characteristic function $\phi_n(t)$ has the limit $\phi(t)$ for every real $t$ and this limit is continuous at $t=0.$
 Therefore by Lemma \ref{continuity theorem}, we get the required Theorem \ref{distribution of additive function over function field introduction}.

\vspace{2mm}
\noindent
In the case of monic irreducible polynomials, the distribution function is
\[
\widetilde{F}_n(x):=\frac{1}{|\mathcal{P}_{n}|} \nu_n\big\{P; \psi_1(P+h_1)+\psi_2(P+h_2)\leq x\big\}
\] 
and the corresponding characteristic function is 
 \[
 \widetilde{\phi}_n(t)=\frac{1}{|\mathcal{P}_{n}|}\sum_{P\in \mathcal{P}_{n}}\exp \left(it\left(\psi_1(P+h_1)+\psi_2(P+h_2)\right)\right).
 \]
Following a similar argument as above, we complete the proof.. 

\vspace*{2mm} 
\noindent \textbf{Acknowledgements:}  We are grateful to Oleksiy Klurman  for insightful comments and corrections. We thank the anonymous referees for their valuable comments and insightful suggestions that have improved the quality of the manuscript.


\begin{thebibliography}{0}

\bibitem{AND1}
J. C. Andrade, L. Bary-Soroker and Z. Rudnick, {Shifted convolution and the
Titchmarsh divisor problem over $F_q[t]$}, {\it Phil.Trans. R. Soc. A}, {\bf 373}: 20140308.

\bibitem{AND2}
J. C. Andrade, A.Shamesaldeen and C. Summersby, {On elementary estimates of arithmetic sums for polynomial rings over finite fields}, {\it J. Number Theory}, {\bf 199} (2019), 49-62.

\bibitem{CAR}
D. Carmon, {The autocorrelation of the M\"{o}bius function and Chowla’s conjecture for the rational function field in characteristic 2}, {\it Philos Trans A Math Phys Eng Sci.}, {\bf 373} (2015): 20140311. 

\bibitem{CR}
D. Carmon and Z. Rudnick, {The autocorrelation of the M\"{o}bius function and Chowla’s conjecture for the rational function field}, {\it Q. J. Math.}, {\bf 65} (2014), 53–61.

\bibitem{DAR}
P. Darbar, {Triple correlations of multiplicative functions}, {\it Acta Arithmetica}, {\bf 180} (2017), 63-88.

\bibitem{ELL} P.D.T.A. Elliott,
{Probabilistic Number Theory I},
{\it Springer}, {\bf 239}, 1979.

\bibitem{GORO}
O. Gorodetsky, {Mean values of arithmetic functions in short intervals and in arithmetic progressions in the large-degree limit}, {\it Mathematika}, {\bf 66} (2020), 373-394.

\bibitem{GHS}
A. Granville, A. J. Harper and K. Soundararajan, {Mean value of multiplicative functions over function fields}, {\it  Res. number theory}, {\bf 25} (2015).

\bibitem{GHS1}
 A. Granville, A. J. Harper and K. Soundararajan, {A new proof of Halász's theorem, and its consequences}, {\it Compos. Math.}, {\bf  155} (2019), 126-163.
 
 \bibitem{GHS2}
 A. Granville, A. J. Harper and K. Soundararajan,
{A more intuitive proof of a sharp version of Halász's theorem}, {\it Proc. Amer. Math. Soc.}, {\bf 146} (2018), 4099-4104.

\bibitem{GS1} A. Granville and K. Soundararajan
 \emph{Large character sums: pretentious characters and the P\'{o}lya-Vinogradov theorem}, 
{\it J. Amer. Math. Soc.}, {\bf 20} (2007), pp. 357-384.

\bibitem{HSU}
H. Chih-Nung, {The Brun-Titchmarsh theorem in function fields}, {\it J. Number Theory}, {\bf 79} (1999), 67-82.

\bibitem{KAT} I. K\'{a}tai
{On the distribution of Arithmetical functions},
{\it Acta Mathematica Academiae Scientiarum Hungaricae}, {\bf 20} (1969), pp. 60-87.

\bibitem{KLR}
O. Klurman, {Mean values and correlations of multiplicative
functions: The ``pretentious" approach}, {\it Ph.D Thesis}, 2017.

\bibitem{KLR1}
O. Klurman, {Correlations of multiplicative functions and applications}, {\it Compos. Math.}, {\bf 153} (2017), 1622-1657.

\bibitem{KMT}
O. Klurman, A. P. Mangerel, and J. Ter\"{a}v\"{a}inen, {Correlations of multiplicative functions in function fields}, {\it arXiv:2009.13497}.

\bibitem{KMT2}
O. Klurman, A. P. Mangerel, and J. Ter\"{a}v\"{a}inen, {Beyond the Erd\"{o}s discrepancy problem in function fields}, {\it 	arXiv:2202.10370}.

\bibitem{KNP} John Knopfmacher,
{Abstract analytic number theory}, {\it Elsevier Science}, {\bf 12} (2009).

\bibitem{KZ}
John  Knopfmacher and Wen-Bin Zhang,  {Number theory arising from finite fields}, Analytic and probabilistic theory, {\it Monographs and Textbooks in Pure and Applied Mathematics}, {\bf 241}, Marcel Dekker, Inc., New York, 2001.

\bibitem{Mangerel} A. P. Mangerel, {On the bivariate Erdős-Kac theorem and correlations of the M\"{o}bius function}, {\it Math. Proc. Cambridge Philos. Soc.},  {\bf 169} (2020), pp. 547-605. 

\bibitem{MR}
K. Matom\"{a}ki and M. Radziwill,  {Multiplicative functions in short intervals}, {\it  Ann. of Math. (2)}, {\bf 183} (2016), 1015-1056.

\bibitem{MRT1}
K. Matom\"{a}ki, M. Radziwill, T. Tao, {An averaged form of Chowla’s conjecture}, {\it Algebra and  Number Theory}, {\bf 9} (2015), 2167-2196.

\bibitem{ROS}
M. Rosen, {Number theory in function fields}, {\it Graduate Texts in Mathematics}, 2002.

\bibitem{SORO}
L. Bary-Soroker,  {Hardy–Littlewood Tuple Conjecture Over Large Finite Fields}, {\it International Mathematics Research Notices}, {\bf 2014} (2014), 568-575.

\bibitem{SS}
W. Sawin and M. Shusterman, {On the Chowla and twin primes conjectures over $\mathbb{F}_q[T]$}, {\it arXiv:1808.04001} (2018).

\bibitem{ST4} G. Stepanauskas,
\emph{Mean values of Multiplicative Functions III},
 {\it New trends in probability and statistics}, {\bf 4} (1997), pp. 371-387.

\bibitem{Tao}
T. Tao, {The logarithmically averaged Chowla and Elliott conjectures for two-point correlations},
{\it Forum Math. Pi}, {\bf 4} (2016).

\bibitem{TEN}
G. Tenenbaum,
{Introduction to Analytic and Probabilistic Number Theory},
{\it Cambridge University Press}, 1995.

\bibitem{TT}
T. Tao and J. Ter\"{a}v\"{a}inen,
{The structure of logarithmically averaged correlations of multiplicative functions, with applications to the Chowla and Elliott conjectures},  
 {\it  Duke Math. J.}, {\bf 168} (2019), 1977-2027.
 
 \bibitem{TT1}
 T. Tao and J. Ter\"{a}v\"{a}inen,
 {The structure of correlations of multiplicative functions at almost all scales, with applications to the Chowla and Elliott conjectures}, {\it Algebra and Number Theory}, {\bf 13} (2019), 2103-2150.
 
\bibitem{WEBB} William A. Webb, {Sieve methods for polynomial rings over finite fields}, {\it J. Number Theory}, {\bf 16} (1983), pp. 343-355.


\end{thebibliography}
\end{document}